\documentclass[11pt,reqno]{amsart}
\usepackage[foot]{amsaddr}
\usepackage{amsmath} %[intlimits]
\usepackage{amsfonts,amssymb,amsthm,bm}
\usepackage{amsrefs}
\usepackage{esint}

\usepackage{enumitem}
\usepackage[utf8]{inputenc}
\usepackage[T1]{fontenc}
\usepackage{stackrel}

\usepackage{fullpage}

\usepackage[dvipsnames]{xcolor}

\usepackage{hyperref}

\newcommand{\ud}[0]{\,\mathrm{d}}

\newcommand{\op}[1]{\operatorname{#1}}

\newcommand\abs[1]{\left|#1\right|} % absolute value
\DeclareMathOperator{\diam}{diam} % diameter of a set
\DeclareMathOperator{\dist}{\mathit{d}} % distance
 \DeclareMathOperator{\lip}{Lip}

\newcommand{\Chi}{\mathcal{X}} % characteristic function
\let\emptyset\varnothing % nicer empty set
\DeclareMathOperator*{\loc}{loc} % local
\newcommand{\aabs}[1]{|#1|}

\newcommand{\Babs}[1]{\Big|#1\Big|}

\newcommand{\Norm}[2]{\|#1\|_{#2}}

\newcommand{\ave}[1]{\langle #1\rangle}

\newcommand{\osc}[0]{\operatorname{osc}}

\newcommand{\calK}{\mathcal{K}}

\newcommand{\calQ}{\mathcal{Q}}

\newcommand{\calO}{\mathcal{O}}

\newcommand{\vare}{\varepsilon}

\newcommand{\N}{\mathbb{N}}
\newcommand{\Z}{\mathbb{Z}}

\newcommand{\R}{\mathbb{R}}
\newcommand{\C}{\mathbb{C}}

\newcommand{\BMO}{\mathrm{BMO}}
\newcommand{\VMO}{\mathrm{VMO}}
\newcommand{\VC}{\dot{\operatorname{VC}}}
\newcommand{\V}{\dot{\operatorname{V}}}

\def\XXint#1#2#3{\mkern3mu{\setbox0=\hbox{$#1{#2#3}{\int}$ }
\vcenter{\hbox{$#2#3$ }}\kern-.6\wd0}}

\theoremstyle{plain}
\newtheorem{theorem}{Theorem}

\newtheorem{lemma}[theorem]{Lemma}
\newtheorem{corollary}[theorem]{Corollary}
\newtheorem{example}[theorem]{Example}

\numberwithin{theorem}{section}

\newtheorem{remark}[theorem]{Remark}

\theoremstyle{definition}
\newtheorem{dfn}[theorem]{Definition}

\title{Approximation in H\"{o}lder spaces}

\author[Carlos Mudarra]{Carlos Mudarra}
\address[C.M.]{Department of Mathematical Sciences, Norwegian University of Science and Technology, 7941 Trondheim, Norway.}
\email{carlos.mudarra@ntnu.no}

\author[Tuomas Oikari]{Tuomas Oikari}
\address[T.O.]{Departament de Matem\`atiques, Universitat Aut\`onoma de Barcelona,
	Edifici C Facultat de Ci\`encies, 08193 Bellaterra (Barcelona), Catalonia}
\email{tuomas.oikari@gmail.com}

\date{\today}



\begin{document}

\maketitle

\begin{abstract}
	We introduce new vanishing subspaces of the homogeneous H\"{o}lder space $\dot{C}^{0,\omega}(X,Y)$ in the generality of a doubling modulus $\omega$ and normed spaces $X$ and $Y.$ For many couples $X,Y,$ we show these vanishing subspaces to completely characterize those H\"older functions that admit approximations, in the H\"{o}lder seminorm, by smooth, Lipschitz and boundedly supported functions. 
	We present connections to bi-parameter harmonic analysis on the Euclidean space by providing
	applications to the compactness of the bi-commutator of two Calder\'on-Zygmund operators
	
	%\noindent \textit{Key words:} smooth approximation, bounded support approximation, vanishing oscillation, H\"older space, Banach space 

	%\medskip

	%\noindent \textit{2020 Mathematics Subject Classification:} 46T20, 41A29, 42B35, 46B20, 26B35
\end{abstract}

\section{Introduction and main results}
\subsection{Introduction}
We consider normed spaces $X, Y$ and moduli of continuity $\omega$ 
and study the H\"{o}lder spaces $\dot{C}^{0,\omega}(X,Y)$ that consist of those functions $f:X \to Y$ for which
\[
\Norm{f}{\dot{C}^{0,\omega}(X,Y)} = \sup_{ \substack{x,y\in X \\ x \neq y}}\frac{\| f(x)-f(y)\|_Y}{\omega(\| x-y\|_X)}<\infty.
\]
We establish various full characterizations of approximability in the class $\dot{C}^{0,\omega}(X,Y)$ by Lipschitz, smooth and boundedly supported functions in terms of three vanishing subspaces of $\dot{C}^{0,\omega}(X,Y)$ that demand vanishing behaviour of the H\"older norm of functions on \emph{small, large} and \emph{far} away scales. We will define these subspaces soon, but before that we make a few general observations.

First of all, if $Y$ is a Banach space, then so is $\dot{C}^{0,\omega}(X,Y),$ when equipped with the norm
\begin{equation}\label{eq:BanachNorm}
|||f|||_{C^{0,\omega}(X,Y)}=\|f\|_{\dot{C}^{0,\omega}(X,Y)} + \|f(0)\|_Y.
\end{equation}
Dropping the term $ \|f(0)\|_Y,$ the seminorm $\| \cdot \|_{\dot{C}^{0,\omega}(X,Y)}$ becomes a norm when the functions are identified with equivalence classes modulo additive constants.
In the particular case $\omega(t)= t^ \alpha,$ $\alpha\in (0,1),$ we recover the H\"older spaces $\dot{C}^{0,\alpha}(X,Y).$ 
Moreover, provided $\omega$ is any reasonable modulus of continuity -- at least $\omega(t)\to 0,$ as $t\to 0$ -- the spaces $\dot C^{0,\omega}(X,Y)$ become subspaces of uniformly continuous functions.

The theory of \textit{uniform} approximation of continuous functions by smooth functions has been extensively studied and developed in the last decades.  For instance, when $X$ is a Hilbert space, Lasry and Lions \cite{LasryLions} showed that Lipschitz real-valued functions can be uniformly approximated by $C^{1,1}$ Lipschitz functions. If $X$ is moreover separable, then a theorem of Moulis \cite{Moulis}  allows these approximations to be upgraded to $C^\infty$ approximations.  Later, Cepedello-Boiso \cite{Cepedello} obtained several results concerning $C^1$ or $C^{1,\alpha}$ approximation of uniformly continuous functions in super-reflexive Banach spaces $X$ --  see also our Theorem \ref{ThmApproxsuper-reflexive} below and the very recent work of Johanis \cite{Jo24}, that contribute to this line of research.  Then, H\'{a}jek and Johanis \cite{HajJoh10} proved that when $X$ is separable with a $C^k$ smooth and Lipschitz bump, and either $X$ or $Y$ is super-reflexive, then  uniformly continuous mappings admit approximations by $C^k$ maps that are still uniformly continuous. Lastly, we mention the real-analytic approximations of Azagra, Fry, and Keener \cite{AzFrKe12} when $X$ is a Banach space with a separating polynomial; then, Lipschitz real-valued functions on $X$ can be uniformly approximated by real analytic Lipschitz functions. For more background references on smoothness and renorming on Banach spaces, and the related approximation results, we refer to the monographs \cite{BenyaminiLindenstrauss}, by Benyamini and Lindenstrauss; \cite{DevilleGodefroyZizler}, by Deville, Godefroy, and Zizler; and \cite{HajJoh10}, by H\'{a}jek and Johanis.
 
In the present article, we prove full characterizations of approximation with H\"older type rate of convergence taking the place of uniform convergence. We note that these two types of convergence do not follow from each other; see Appendix \ref{appendix:comparison} for a detailed comparison with simple observations and examples. 
Our main result is best described as the recognition of the vanishing scales \emph{small, large} and \emph{far} as precisely the concept that allows a complete description of the closure of smooth, Lipschitz and boundedly supported functions with respect to the $\dot{C}^{0,\omega}$ seminorm. We notate
\[
\osc^{\omega}_{\delta}(f) =  \sup_{ \substack{x\not=y\in X \\ \| x-y\|_X = \delta}}\frac{ \| f(x)-f(y)\|_Y}{\omega(\| x-y \|_X)}
,\qquad 
\osc^{\omega}_{(x,y)}(f) = \frac{ \| f(x)-f(y)\|_Y}{\omega(\| x-y \|_X)},
\]
and then define these three vanishing scales as 
\begin{align*}
	\V^{0,\omega}_{\op{small}}(X,Y) &= \big\{  f:X\to Y \, : \, \lim_{\delta \to 0} \osc^{\omega}_{\delta}(f)  = 0 \big\},\, \\
	\V^{0,\omega}_{\op{large}}(X,Y) &= \big\{ f:X\to Y \, : \, \lim_{\delta \to \infty} \osc^{\omega}_{\delta}(f) = 0  \big\}, \\
	\V^{0,\omega}_{\op{far}}(X,Y) &= \big\{ f:X\to Y \, : \, \lim_{\delta \to \infty} \sup_{\min(\|x\|,\|y\|) >\delta} \osc^{\omega}_{(x,y)}(f)  = 0 \big\}.
\end{align*}
\begin{dfn}\label{def:VCome}  For each scale $\Gamma\in\{\op{small}, \op{large}, \op{far}\},$ define
	\begin{align*}
		\VC^{0,\omega}_{\Gamma}(X,Y) &= \dot{\op{V}}^{0,\omega}_{\Gamma}(X,Y) \cap \dot C^{0,\omega}(X,Y), \\
		\VC^{0,\omega}(X,Y) &= \VC^{0,\omega}_{\op{small}}(X,Y) \cap \VC^{0,\omega}_{\op{far}}(X,Y) \cap \VC^{0,\omega}_{\op{large}}(X,Y).
	\end{align*}
\end{dfn}
If $Y$ is a Banach space, then so are all of $\VC^{0,\omega}_{\Gamma}(X,Y)$ as spaces of modulo additive constant equivalence classes; for full details, see Remark \ref{remark:VCareclosedsubspaces} below. From now on, whenever no confusion is possible, we start dropping the source and target spaces from the norms and simplify notation with $\| x \|_X = \|x\|$ and $\| f(x) \|_Y = |f(x)|.$ In the same way, we write $\dot{C}^{0,\omega}(X,Y) = \dot{C}^{0,\omega}(X)= \dot{C}^{0,\omega},$ etc.

We assume our moduli $\omega$ to be non-decreasing and satisfy
\begin{align}
	&\lim_{t\to 0}\omega(t) = 0,\qquad \lim_{t\to \infty}\omega(t) = \infty, \label{eq:mod:coer0} \\
	&\lim_{t \to 0} t/\omega(t)=0, \label{eq:mod:coer1} \\
	&\omega(2t)\leq C_{\op{db}}\omega(t),\qquad \mbox{for}  \qquad t\in (0,\infty) \label{eq:mod:db},
\end{align}
where the doubling constant satisfies $C_{\op{db}}>1.$  
This class of moduli includes the H\"older moduli $\omega(t) := t^{\alpha},$ when $\alpha\in(0,1),$ but also other moduli that behave as $1/\log(1/t)$ close to the origin, which appear in the study of regularity for some partial differential equations, see e.g. Liao \cite{LIAO2022}. 
Our results involve several assumptions on the moduli with different theorems requiring slightly different assumptions, see Table \ref{table} below for a summary.

The classes $\VC^{0,\alpha}_{\op{small}}$ are also called the \textit{little H\"older spaces}. These are important  in the study of Lipschitz algebras in metric spaces, and we refer to the monograph by Weaver \cite[Chapters 4 and 8]{Weaverbook} for a background. See also the recent monograph by D. Mitrea, I. Mitrea, and M. Mitrea \cite[Chapter 3]{Mitreabook} for a detailed exposition of the H\"older spaces in the setting of Ahlfors regular sets and their connection with functions of vanishing mean oscillation.

\subsection*{Notation}

We denote $A \lesssim B$, if $A \leq C B$ for some constant $C>0$ depending only on an underlying function space that is fixed relative to ongoing considerations. 
For example, $C$ could depend on the modulus $\omega,$ or the dimension $n$ in the case $X=\R^n.$
Then $A \sim B$, if $A \lesssim B$ and $B \lesssim  A.$ Subscripts or variables on constants and quantifiers $C_{a},C(a),\lesssim_{a}$ indicate their dependence on those subscripts.

\subsection{Main results}   
We begin by stating a result without any smoothness on approximations by functions with bounded support. Many of our later results build on this one.

\begin{theorem}\label{thm:approx:bs} Let $X,Y$ be normed spaces, the modulus $\omega$ satisfy \eqref{eq:mod:coer0}, \eqref{eq:mod:coer1} and \eqref{eq:mod:db}.
	Then, there holds that 
	\begin{align*}
		\VC^{0,\omega}(X,Y) =  \overline{\VC^{0,\omega}_{\op{small}}(X,Y) \cap C_{\op{bs}}(X,Y)}^{\dot C^{0,\omega}(X,Y)}.
	\end{align*}
\end{theorem}

Here and below, for any $\mathcal{C}(X,Y) \subset \dot C^{0,\omega}(X,Y),$ we denote by $\overline{\mathcal{C}(X,Y)}^{\dot C^{0,\omega}(X,Y)}$ the closure of 
$\mathcal{C}(X,Y)$ under the $\dot{C}^{0,\omega}(X,Y)$ seminorm. We refer here to the family of those $f\in \dot{C}^{0,\omega}(X,Y)$ for which there exists a sequence $\lbrace f_n \rbrace_n \subset \mathcal{C}(X,Y)$ so that $\lim_n \|f_n-f\|_{\dot{C}^{0,\omega}(X,Y)} =0.$ The subscripted set $\mathcal{C}_{\op{bs}}(X,Y)$ consists of boundedly supported functions in $\mathcal{C}(X,Y).$ In Theorem \ref{thm:approx:bs} $C_{\op{bs}}(X,Y)$ denotes continuous functions with bounded support. In particular, the functions in $\VC^{0,\omega}_{\op{small}}(X,Y)$ are continuous, and the intersection with $C_{\op{bs}}(X,Y)$ gives the approximations bounded support.

While Theorem \ref{thm:approx:bs} does not provide any smoothness, it turns out to be a crucial step in our other theorems. For instance, letting $X= \R^n$ and $Y$ be a Banach space, Theorem \ref{thm:approx:bs} coupled with a mollification argument gives  $C^\infty_{\op{bs}}(\R^n,Y)$ approximations, this is Theorem \ref{thm:main:Euc} below.  Both of these theorems will be proved in Section \ref{sect:approx:bs}.  

Our results are not confined to functions with a finite-dimensional source $X,$ and letting $Y= \R,$ we prove smooth and Lipschitz approximation for the classes $ \VC^{0,\omega}_{\op{small}}(X,\R) $ and $\VC^{0,\omega}(X,\R).$ (The corresponding approximations for $\C^n$-valued functions in place of $\R$-valued then follow after splitting into the coordinate functions and then each of these to the real and imaginary parts.)  In particular, with $X$ being a separable normed space, we have the following general result.   
\begin{theorem}\label{thm:approx:C^k}
	Let $k\in \N \cup\lbrace \infty \rbrace$ and $X$ be a separable normed space admitting a $C^{k}$ and Lipschitz bump function. Then, for a modulus $\omega$ satisfying  \eqref{eq:mod:coer0}, \eqref{eq:mod:coer1} and \eqref{eq:mod:db}, the following hold:
	\begin{itemize}
		\item[(i)] $ \overline{C^k(X)\cap \lip(X) \cap \dot{C}^{0,\omega}(X)}^{\dot{C}^{0,\omega}(X)}=\VC_{\op{small}}^{0,\omega}(X),$
		\item[(ii)] $ \overline{C^k_{\op{bs}}(X)\cap \mathrm{Lip}(X)}^{\dot{C}^{0,\omega}(X)}=\VC^{0,\omega}(X).$
	\end{itemize}
\end{theorem}

Notice that, together with the $C^k$ order of smoothness, Theorem \ref{thm:approx:C^k} guarantees Lipschitz approximations. In particular, any separable normed space $X$ with an equivalent norm of class $C^k$ satisfies the assumption of Theorem \ref{thm:approx:C^k}, see Remark \ref{remark:smoothnessofLpspaces}. Moreover, when $X$ is a separable Hilbert space, one can arrange $C^\infty$ smooth approximations, thus obtaining an infinite dimensional version of Theorem \ref{thm:main:Euc}, see Corollary \ref{CorollaryHilbertSeparableApprox}.
We will restate Theorem \ref{thm:approx:C^k} and prove it in Subsection \ref{subsect:LipSmoApprox}.

\smallskip

Another class of spaces $X$ on which we obtain smooth approximations are super-reflexive Banach spaces. A theorem of Pisier \cite{Pisier} says that any super-reflexive space admits a renorming with \textit{modulus of smoothness of power type} $1+\alpha,$ for some $\alpha \in (0,1]$ -- see \eqref{C1alpharenorming} below for the precise formulation. These spaces contain the Hilbert (separable or not) spaces, together with many of the classical Banach function spaces, such as the $L^p$ spaces, with $1< p< \infty;$ see Remark \ref{rem:smoothnessL^psuper}.  In the following theorem we establish $C^{1,\alpha}$ and Lipschitz approximations for the classes $\VC^{0,\omega}(X)$ and $\VC^{0,\omega}_{\op{small}}(X).$

\begin{theorem}\label{thm:intro:super-reflexive}
Let $X$ be a super-reflexive space that admits a renorming with modulus of smoothness of power type $1+\alpha,$  $\alpha \in (0,1].$ Then, for a modulus $\omega$ satisfying  \eqref{eq:mod:coer0}, \eqref{eq:mod:coer1} and \eqref{eq:mod:db}, the following hold:
	\begin{itemize}
		\item[(i)] $\overline{C^{1,\alpha}(X) \cap \mathrm{Lip}(X) \cap \dot{C}^{0,\omega}(X)}^{\dot{C}^{0,\omega}(X)} = \VC^{0,\omega}_{\op{small}}(X),$
		\item[(ii)] $ \overline{C^{1,\alpha}_{\op{bs}}(X)}^{\dot{C}^{0,\omega}(X)} =\VC^{0,\omega}(X).$
	\end{itemize}
\end{theorem}

In Remark \ref{rem:smoothnessL^psuper}, we will see that in Hilbert spaces and in $L^p$ spaces, with $p \geq 2,$ Theorem \ref{thm:intro:super-reflexive} yields $C^{1,1}$ approximations, i.e., $C^1$ functions with \textit{Lipschitz} first derivatives. 

Finally, for mappings $f:X \to Y,$ where $Y$ is an arbitrary Banach space, in Subsection \ref{subsect:C0GamApprox} we obtain $C^\infty$ smooth approximations in the case where $X$ is a space of the form $X=c_0(\mathcal{A}).$ 
\begin{theorem}\label{thm:intro:c0YApprox}
For an arbitrary set of indices $\mathcal{A},$ let $X=c_0(\mathcal{A}),$ and let $Y$ be a Banach space. Then, for any modulus $\omega$ satisfying  \eqref{eq:mod:coer0}, \eqref{eq:mod:coer1} and \eqref{eq:mod:db}, the following hold:
	\begin{itemize}
		\item[(i)] $\overline{C^{\infty}(X,Y)\cap \VC^{0,\omega}_{\op{small}}(X,Y)}^{\dot{C}^{0,\omega}(X,Y)}= \VC^{0,\omega}_{\op{small}}(X,Y),$
		\item[(ii)] $\overline{C^{\infty}_{\op{bs}}(X,Y)\cap \VC^{0,\omega}(X,Y)}^{\dot{C}^{0,\omega}(X,Y)}=\VC^{0,\omega}(X,Y).$
	\end{itemize}
\end{theorem}

Theorem \ref{thm:intro:c0YApprox} will be restated and proved in Subsection \ref{subsect:C0GamApprox}, see Theorem \ref{thm:main:c0YApprox}. There, we will also show that the approximations can be taken to be Lipschitz in the case $Y=\R.$

	\subsection{Applications on the Euclidean space}\label{subsect:applicationsEucld} Above we presented results that guarantee approximations provided we know pointwise vanishing type conditions. When $X=\R^n,$ we will formulate these results in terms of mean oscillation type conditions, which are very convenient when studying commutators of Calderón-Zygmund operators (CZOs), for example.
	In the one parameter setting our results recover known components of the theory, but from a more general approach.
	In the bi-parameter context applications to the study of the compactness of commutators were recently presented in \cites{MarOik24}.   We describe all of this in full detail below, but let us begin by a motivation.
	
	For many non-degenerate CZOs $T$ on $\R^n,$ the symbol of the commutator $b\in L^1_{\loc}(\R^n,\C)$ and $p,q\in (1,\infty),$ there holds that 
	\begin{align}\label{eq:comm:bdd}
		\|[b,T]\|_{L^p(\R^n)\to L^q(\R^n)} \sim	\| b\|_{X^{p,q}(\R^n)},
	\end{align} 
	where the space $X^{p,q}(\R^n)$ is determined as follows. Given $p,q\in (1,\infty),$ let $\alpha(p,q)$ and $r(p,q)$ be defined through
	$\alpha(p,q)/n=1/p-1/q$ and $1/q=  1/r(p,q)+ 1/p.$
	Then, $X^{p,q}(\R^n)$ is the space of bounded mean oscillations $\BMO(\R^n),$ when $p=q;$ the (fractional BMO) space $\BMO^{\alpha(p,q)}(\R^n),$ when $p<q;$ and the Lebesgue space $L^{r(p,q)}(\R^n)$ modulo additive constants $\dot{L}^{r(p,q)}(\R^n),$ when $p>q.$ (The definitions of the spaces $\BMO,\BMO^{\alpha}$ can be read from Definition \ref{def:BMOX} below with the $``\mbox{moduli}"$ $ t \mapsto t^{0}$ and $t\mapsto t^{\alpha},$ respectively.)
	With emphasis on the recognition of the correct spaces $X^{p,q}(\R^n)$ to which the symbol belongs, the result \eqref{eq:comm:bdd} is due to Nehari \cite{Nehari1957} ($n=1$) and Coifman, Rochberg, Weiss \cite{CRW} ($n\geq 2$), when $q=p;$ due to Janson \cite{Jan1978}, when $q>p;$ and due to Hyt\"{o}nen \cite{HyLpLq}, when $q<p.$ We direct the reader to \cite{HyLpLq} for a further discussion and precise definitions.
	
	%\subsubsection{Off-diagonal compactness of commutators}
	When  $\alpha(p,q)>1,$ the space $\dot C^{0,\alpha(p,q)}(\R^n)$ consists of constant functions and the commutator is bounded iff it is the zero operator (hence compact, in particular). Thus, only those $p,q$ that result in $\alpha(p,q)\leq 1$ are interesting to study. 
	When $\alpha(p,q) = 1$,  \cite[Theorems 1.7. \& 1.8.]{GHWY21} show that the commutator is compact iff it is the zero operator. Then, for $p,q\in (1,\infty)$ in the non-trivial range $\alpha(p,q)<1,$ and for many non-degenerate CZOs $T$,  there holds that 
	\begin{align}\label{eq:comp:approx}
		[b,T]\in \calK\left( L^{p}(\R^n), L^q(\R^n)\right) \Leftrightarrow b\in Y^{p,q}(\R^n) := \overline{C^{\infty}_c(\R^n)}^{X^{p,q}(\R^n)}.
	\end{align}
	The characterization  \eqref{eq:comp:approx} is due to Uchiyama \cite{Uch1978}, when $p=q$; due to Guo, He, Wu and Yang \cite{GHWY21}, when $q>p;$ and due to Hyt\"{o}nen, Li, Tao and Yang \cite{HLTY2023}, when $q<p.$  
	We recall that a result of Meyer's \cite{Mey1964} states that $\BMO^{\alpha(p,q)}(\R^n) = \dot{C}^{0,\alpha(p,q)}(\R^n),$ for $\alpha(p,q)\in (0,1).$  Thus, the following Theorem \ref{thm:main:Euc} provides in the case $\alpha(p,q)\in (0,1)$ a new characterization of the space $Y^{p,q}(\R^n).$
	\begin{theorem}\label{thm:main:Euc} 
		Let $Y$ be a Banach space and the modulus $\omega$ satisfy \eqref{eq:mod:coer0}, \eqref{eq:mod:coer1} and \eqref{eq:mod:db}.
		Then,
		\begin{align*}
			\VC^{0,\omega}(\R^n, Y)  =  \overline{C^{\infty}_c(\R^n, Y)}^{\dot C^{0,\omega}(\R^n, Y)}.
		\end{align*}
	\end{theorem}
	To obtain the implications $``\Rightarrow"$ in \eqref{eq:comp:approx} it is useful to have a description of $Y^{p,q}(\R^n),$ in the cases $q\geq p,$ in terms of vanishing mean oscillation (VMO) type criteria. Such VMO criteria (Definition \ref{def:VMOX}, below) usually follow immediately from the existence of approximations, but to establish approximations from VMO type criteria is delicate. Recall that in the previous section we established the existence of approximations beginning from pointwise vanishing criteria $\VC.$  Below, as Theorem \ref{thm:cubepw}, we provide the correspondence between these $\VMO$ and the $\VC$ type criteria, which allows us, when $X=\R^n,$ to reformulate all of our results, among them Theorem \ref{thm:main:Euc}, in terms of mean oscillation conditions.
	
	\subsection{Formulation as mean oscillation conditions}\label{section:pwOsc}
	\begin{dfn}\label{def:BMOX}  
		Let $\omega$ be a modulus and $Y$ a Banach space. Then, $\BMO^{\omega}(\R^n,Y)$ consists of those $f\in L^1_{\loc}(\R^n,Y)$ locally Bochner integrable functions for which 
		\[
		\Norm{f}{\BMO^{\omega}(\R^n)} = \sup_{Q} 	\calO^{\omega}(f;Q) < \infty,\qquad \calO^{\omega}(f;Q) =   \frac{1}{\omega(\ell(Q))}\fint_Q |f-\ave{f}_Q|_Y,
		\]
		where the supremum is taken over all cubes in $\R^n.$
	\end{dfn}
We direct the reader to Section \ref{sect:Bochner} below for a recap of Bochner integrability.
	\begin{dfn}\label{def:VMOX} 	Let $\omega$ be a modulus and $Y$ a Banach space and define the vanishing scales
		\begin{align*}
			\VMO^{\omega}_{\op{small}}(\R^n,Y) &= \Big\{  f\in \BMO^{\omega}(\R^n,Y) \, : \, \lim_{\delta \to 0}  \sup_{\substack{ \ell(Q) = \delta}} 	\calO^{\omega}(f;Q) = 0  \Big\}, \\
			\VMO^{\omega}_{\op{large}}(\R^n,Y) &= \Big\{ f\in \BMO^{\omega}(\R^n,Y) \, : \, \lim_{\delta \to \infty}  \sup_{\substack{\ell(Q) = \delta}} 	\calO^{\omega}(f;Q) = 0 \Big\}, \\
			\VMO^{\omega}_{\op{far}}(\R^n,Y) &= \Big\{  f\in \BMO^{\omega}(\R^n,Y) \, : \, \lim_{\delta \to \infty}  \sup_{\substack{ \dist(Q,0) >\delta } }	\calO^{\omega}(f;Q) = 0 \Big\}, 
		\end{align*}
		where the suprema are taken over all cubes. Then, define 
		\begin{align*}
			\VMO^{\omega}(\R^n,Y) &= \VMO^{\omega}_{\op{small}}(\R^n,Y) \cap \VMO^{\omega}_{\op{far}}(\R^n,Y) \cap \VMO^{\omega}_{\op{large}}(\R^n,Y). 
		\end{align*}
	\end{dfn}

	The approximation $\VMO(\R^n) =  \overline{C^{\infty}_c(\R^n)}^{\BMO(\R^n)}$ was already present in the work of Uchiyama \cite{Uch1978}.  Uchiyama's construction relies on covering any bounded neighbourhood of the origin by disjoint cubes of equal shared side length $l>0$. The key property used is that there exists some $M>0$ such that given any cube of side length $\sim l$, it intersects at most $M$ many of these cubes.  Uchiyama's construction does not easily translate to the case $\alpha\in (0,1)$ and in \cite{GHWY21} an approximation for $\VMO^{\alpha}(\R^n) =  \overline{C^{\infty}_c(\R^n)}^{\BMO^{\alpha}(\R^n)}$ was provided that relies on the same covering property that Uchiyama's original argument did.  Clearly such a property fails, when $X$ is an arbitrary normed space.
Therefore, one more fact of Theorem \ref{thm:main:Euc}, in particular of its proof, is that it provides a conceptually easier approximation that uses only basic mollification along with a very carefully chosen truncation of the support. To make the connection with the mean oscillations and the pointwise conditions, we assume the following summability condition
	\begin{align}\label{eq:mod:dini}
		[\omega]_{*} :=\sup_{s>0}\frac{1}{\omega(s)}\int_0^{s}\omega(t)\frac{\ud t}{t}< \infty.
	\end{align}
	\begin{theorem}\label{thm:cubepw} Let $Y$ be a Banach space and the modulus $\omega$ satisfy \eqref{eq:mod:db} and \eqref{eq:mod:dini}.
		Then, 
		\begin{align}\label{eq:meyer}
			\Norm{\cdot }{\dot C^{0,\omega}(\R^n,Y)}  \lesssim_{n} [\omega]_{*} \Norm{\cdot}{\BMO^\omega(\R^n,Y)},
			\qquad  \Norm{\cdot}{\BMO^\omega(\R^n,Y)} \lesssim_{n,C_{\op{db}}}  \Norm{\cdot}{\dot C^{0,\omega}(\R^n,Y)}.
		\end{align}
		Moreover, if in addition  $\omega(\infty) = \infty$, then for each scale $\Gamma\in \{\op{small}, \op{far},\op{large}\},$ we have 
		\begin{align}\label{eq:postM1}
			\VC^{0,\omega}_{\Gamma}(\R^n,Y) = \VMO^{\omega}_{\Gamma}(\R^n,Y),\qquad \VC^{0,\omega}(\R^n,Y) = \VMO^{\omega}(\R^n,Y).
		\end{align} 
	\end{theorem}
	The proof idea of Theorem  \ref{thm:cubepw} goes back to Meyers \cite{Mey1964}, who obtained a version of the bounds \eqref{eq:meyer} with the modulus $t^{\alpha},$ for $\alpha>0.$ 
	The condition \eqref{eq:mod:dini} appears new, but we recognize that it is a natural summability condition, after all. The identification \eqref{eq:postM1} is new.
	Combining Theorems \ref{thm:main:Euc} and \ref{thm:cubepw}, we immediately obtain the following.
	
	\begin{theorem}\label{thm:main:Euc3} Let $Y$ be a Banach space and the modulus $\omega$ satisfy \eqref{eq:mod:coer0}, \eqref{eq:mod:coer1},  \eqref{eq:mod:db} and \eqref{eq:mod:dini}.
		Then, 
		\begin{align*}
			\VMO^{\omega}(\R^n,Y)  =  \overline{C^{\infty}_c(\R^n,Y)}^{\BMO^{\omega}(\R^n,Y)}.
		\end{align*}
	\end{theorem}

In the following table we gather the different background assumptions on modulus used in the theorems above. 
		\begin{table}[h]		 
				\caption{Conditions on the modulus in Theorems \ref{thm:approx:bs} through \ref{thm:main:Euc3}}
		\begin{tabular}{ c | c | c | c | c  | c | c | c  }\label{eq:table}
			\text{Theorem}     &   \ref{thm:approx:bs}                & \ref{thm:approx:C^k} & \ref{thm:intro:super-reflexive} &  \ref{thm:intro:c0YApprox} &  \ref{thm:main:Euc} &  \ref{thm:cubepw}    & \ref{thm:main:Euc3}           \\
			\hline 
			$\omega(0^+) = 0$     & x&            x &           x    &           x&  x&           - &          x      \\ 
			$\omega(\infty) = \infty$                        & x&           x &           x    &           x & x&          x &           x      \\ 
			$\omega$\text{ is doubling}              & x&           x &           x    &           x & x &          x&          x      \\ 
			$\lim_{t \to 0} t/\omega(t)=0$       & x&           x &           x   &           x & x&           -&          x     \\ 
				$[\omega]_* < \infty$            & -&           - &          -    &           - & -&           x &           x   
			%$\sup_{s>0}\tfrac{1}{\omega(s)}\int_{0}^s\tfrac{\omega(t)\ud t}{t} < \infty$            & -&           - &          -    &           - & -&           x &           x   
		\end{tabular}
		 
	%	\caption{Conditions for the modulus of continuity in Theorems \ref{thm:approx:bs}--\ref{thm:main:Euc3}}
		\label{table}
	\end{table}

\subsubsection{Bochner integrability}\label{sect:Bochner}
In the proof of Theorem \ref{thm:main:Euc}, as well as in Theorem \ref{thm:main:c0YApprox}, we will make use of the Bochner integral to build up smooth approximations. We refer to \cite[pp. 99--101]{BenyaminiLindenstrauss} for an exposition on the integrability of Banach-valued maps, and to \cite[pp. 39--40]{HajJoh14} for Differentiation Under the Integral Sign results for the Bochner integral. If $\Omega \subset \R^n$ is Lebesgue measurable, a simple function $s:  \Omega \to Y$ is a function of the form $s = \sum_{i=1}^N y_i \cdot \mathcal{X}_{E_i},$ for measurable subsets $E_1,\ldots, E_N \subset \Omega$, and vectors $y_1,\ldots, y_N \in Y.$ Then, we say that a function $f: \Omega \to Y$ is measurable if it is the a.e. pointwise limit of a sequence $\lbrace s_n \rbrace_n$ of simple funtions in $\Omega.$ Denoting by $\lambda$ the Lebesgue measure in $\R^n$, a simple function $s$ as above is Bochner-integrable in $\Omega$ whenever $\lambda(E_i) < \infty$ for every $i,$ and the Bochner-integral of $s$ is defined by $\int_\Omega s(x) \ud x = \sum_{i=1}^N y_i \lambda(E_i).$ And a measurable function $f: \Omega \to Y$ is Bochner-integrable if there exists a sequence of simple functions $\lbrace s_n \rbrace_n$ converging a.e. to $f$ in $\Omega$ so that
$$
\lim_{n,m \to \infty} \int_\Omega \|f_n(x)-f_m(x)\|_Y \ud x=0.
$$
In such case, the Bochner integral of $f$ is defined by 
$$
\int_\Omega f = \lim_n \int_\Omega s_n.
$$
Lastly, a function $f$ is said to be locally Bochner interable, provided that $ f$ is Bochner integrable on each compact set $E\subset\mathbb{R}^n.$
	
	\subsection{Applications to the compactness of the bi-commutator}\label{sect:appli:bicom}
Theorems \ref{thm:main:Euc}, \ref{thm:cubepw} and \ref{thm:main:Euc3} were partly motivated by applications to bi-parameter harmonic analysis that followed in the subsequent work of the second named author \cite{MarOik24}. We turn to describing the details. 
We consider $\R^n = \R^{n_1}\times\R^{n_2}$ as a bi-parameter space, take CZOs $T_i$ on $\R^{n_i}$ and $b\in L^1_{\loc}(\R^n,\C)$, where $T_i$ acts on functions on $\mathbb{R}^n$ by acting only on the variable $x_i.$  Then, the bi-commutator $[T_2,[b,T_1]]$ is obtained by commuting the symbol first with $T_1$ and then with $T_2.$ 

The product BMO upper bound for the $L^p$-to-$L^p$ boundedness of the bi-commutator was obtained with the Hilbert transforms in Ferguson, Sadosky \cite{FerSad} and in full generality later by Dalenc, Ou \cite{DaO16}. 
The proof of the corresponding lower bound in terms of the product $\BMO$ 
for specific singular integrals (Hilbert/Riesz transforms)
has been reported to have a gap,
see e.g. \cites{Ferguson2002, HTV2021ce, AHLMO2021}, and a fix has not yet been published.  The existence of this gap adds considerable interest to all natural questions involving the bi-commutator.

Instead of looking at $L^p$-to-$L^p$ bounds, the product stucture of the bi-commutator makes it natural to consider mixed $L^{p_1}(L^{p_2})$-to-$L^{q_1}(L^{q_2})$ bounds, where we allow $q_i\not=p_i.$ This was done in \cite{AHLMO2021}. The exponents directly connecting with our work are those $p_i,q_i\in (1,\infty)$ for which $\beta_i := n_i(1/p_i-1/q_i)\geq 0$ and moreover the inequality $\beta_i > 0$ is strict for some $i=1,2,$ which we start to assume here and below.
The function space to characterize the  $L^{p_1}(L^{p_2})$-to-$L^{q_1}(L^{q_2})$ bounds turned out to be
\begin{multline*}
	\| b\|_{\op{BMO}^{\beta_1,\beta_2}(\R^n)}  \\
	= \sup_{R=I_1\times I_2}\ell(I_1)^{-\beta_1}\fint_{I_1}\ell(I_2)^{-\beta_2}\fint_{I_2} |b(x_1,x_2) - \ave{b(x_1,\cdot)}_{I_2} - \ave{b(\cdot,x_2)}_{I_1} + \ave{b}_{I_1\times I_2}|\ud x_2\ud x_1.
\end{multline*}
The space $\op{BMO}^{\beta_1,\beta_2}(\R^n)$ is in fact a $\BMO$-valued or a $\dot C^{0,\alpha}$-valued H\"older space, see \cite[Section 3]{AHLMO2021}.
Then, recently in \cite{MarOik24}, the corresponding $L^{p_1}(L^{p_2})$-to-$L^{q_1}(L^{q_2})$ compactness  was characterized through the VMO analogue $\op{VMO}^{\beta_1,\beta_2}(\R^n).$ Interestingly, such mixed norm compactness results can be used to prove the sufficiency of the natural product VMO condition for compactness in the non-mixed $L^p\to L^p$ case. That proof rests on an adaptation of an extrapolation scheme for commutators from \cite{Oik25}  and the following Theorem \ref{thm:MarOik24}, which is a corollary of the approximations developed in this article.
\begin{theorem}\label{thm:MarOik24} Let $\beta_i\in[0,1)$ and $\beta_i>0$ for some $i=1,2.$ Then,
	$$
	\op{VMO}^{\beta_1,\beta_2}(\R^n) = \overline{C^{\infty}_c(\R^n)}^{\op{BMO}^{\beta_1,\beta_2}(\R^n)}.
	$$
\end{theorem}
Theorem \ref{thm:MarOik24} is a bi-parameter analogue of Theorem \ref{thm:main:Euc3}.
The proof of Theorem \ref{thm:MarOik24} is split between the current article and \cite{MarOik24}. 
We record a quantitative version of Theorem \ref{thm:main:Euc} in the body-text after the proof of Theorem \ref{thm:main:Euc} as Corollary \ref{thm:quant}, and then, assuming Corollary \ref{thm:quant}, Theorem \ref{thm:MarOik24} is proved in \cite[Section 3.]{MarOik24}.

\subsection*{Acknowledgements}
We thank the anonymous referee for thoroughly reading the manuscript and for their comments, which led to improvements in both clarity and presentation.
C. M. was supported by the grant no. 334466 of the Research Council of Norway, ``Fourier Methods and Multiplicative Analysis''. T.O. was supported by the Finnish Academy of Science and Letters, and
by the MICINN (Spain) grant no. PID2020-114167GB-I00.

\section{Bounded support approximation}\label{sect:approx:bs}

In this section we prove Theorem \ref{thm:approx:bs} and then show how to obtain from this Theorem \ref{thm:main:Euc} and then Corollary \ref{thm:quant}. After these, at the end of this section, we prove Theorem \ref{thm:cubepw}. Recall that we simplify the notation by denoting $\| x \|_X = \|x\|$ and $\| f(x) \|_Y = |f(x)|.$  We denote by $B(x,r) : = \lbrace z\in X\, : \, \|x-z\| \leq r \rbrace$ the closed ball with center $x\in X$ and radius $r>0.$
We begin with a simple remark. 
\begin{remark}\label{remark:VCareclosedsubspaces}
	{\em For normed spaces $X$ and $Y,$ and $\Gamma\in\{\op{small}, \op{large}, \op{far}\},$ the set $\VC^{0,\omega}_{\Gamma}(X,Y)$ is closed with respect to the $\dot C^{0,\omega}$-seminorm, meaning that if a sequence of functions $(f_n)_n \subset \VC^{0,\omega}_{\Gamma}(X,Y)$ converges to an $f\in \dot C^{0,\omega}(X,Y),$ then $f\in  \VC^{0,\omega}_{\Gamma}(X,Y)$ as well. Indeed, for any $\varepsilon>0,$ we find $n\in \N$ so that $\|f-f_n\|_{\dot C^{0,\omega}} \leq \varepsilon$. Then, for any two distinct points $x,y\in X,$ we write
		\begin{align*}
			\frac{|f(x)-f(y)|}{\omega(\|x-y\|)} & \leq \frac{|f_n(x)-f_n(y)|}{\omega(\|x-y\|)} + \frac{|(f-f_n)(x)-(f-f_n)(y)|}{\omega(\|x-y\|)}  \\
			&  \leq  \frac{|f_n(x)-f_n(y)|}{\omega(\|x-y\|)}  + \|f-f_n\|_{\dot C^{0,\omega}(X,Y)} \leq \frac{|f_n(x)-f_n(y)|}{\omega(\|x-y\|)}  + \varepsilon.
		\end{align*}
		Since $f_n \in \VC^{0,\omega}_{\Gamma}(X,Y),$ the above clearly implies $f\in  \VC^{0,\omega}_{\Gamma}(X,Y).$

		Consequently, in the case where $Y$ is a Banach space, $(\VC^{0,\omega}_{\Gamma}(X,Y), \|\cdot\|_{\dot C^{0,\omega}(X,Y)})$ is a Banach space of modulo constant equivalence classes.
	}
\end{remark}

The following two lemmas will be very useful.

\begin{lemma}\label{lem:dist=far} Let $X,Y$ be arbitrary normed spaces. Let $\omega$ satisfy \eqref{eq:mod:coer0} and \eqref{eq:mod:db}. Then, we have 
	\begin{align*}
		\VC_{\op{far}}^{0,\omega}(X,Y)  =  \Big\{ f\in  \dot{C}^{0,\omega}(X,Y) \, : \, \lim_{\delta \to \infty} \sup_{\max(\|x\|,\|y\|) >\delta} \osc^{\omega}_{(x,y)}(f)  = 0 \Big\}.
	\end{align*}
\end{lemma}
\begin{proof} Only the inclusion $\subset$ is not immediate; let $f\in   \VC_{\op{far}}^{0,\omega}(X,Y)$ and $\varepsilon>0.$ Let $M = M(\varepsilon)$ be such that if $y,z\in X$ and $\| y\|,\|z\|\geq M,$ then 
	$|f(y)-f(z)| \leq \varepsilon\omega(\|y-z\|).$
	Now we consider arbitrary $x,y\in X,$ and assume $\|y\|> R$, for certain $R \gg M$ to be specified later. If $\|x\|>M,$ then we are done by how we fixed $M$ above. So suppose that $x\in B(0,M) .$ Take a point $z \in [x,y]$ with $\|z\|=M.$ By $\|x\| , \|z\| \leq M$ and $\|y\|, \|z\| \geq M,$ 
	\begin{equation}\label{eq:p}
		\begin{split}
			\frac{|f(x)-f(y)|}{\omega(\|x-y\|)} & \leq \frac{|f(x)-f(z)|}{\omega(\|x-y\|)}+\frac{|f(z)-f(y)|}{\omega(\|z-y\|)} \cdot \frac{\omega(\|z-y\|)}{\omega(\|x-y\|)} \\
			&    \leq \frac{\|f\|_{\dot{C}^{0,\omega}(X)} \omega( \|x-z\|)}{\omega(\|x-y\|)}  +  \varepsilon  \leq \frac{\|f\|_{\dot{C}^{0,\omega}(X)}   \omega( 2M)}{\omega(\|x-y\|)}  +  \varepsilon  .
		\end{split}
	\end{equation}
	To control the above term, note that if $\|y\|>R \gg M,$ then we have $\|x-y\|\geq R/2$ because $\|x\| \leq M.$ Also, by condition $\lim_{t\to\infty}\omega(t)=\infty$, we find $R=R(M,\varepsilon)$ so that $\omega(2M) \leq \varepsilon \omega(R)$. Therefore, since $\|y\| \geq R,$ using the doubling condition \eqref{eq:mod:db} we conclude
	\[
	\op{RHS}\eqref{eq:p}\leq  \frac{\|f\|_{\dot{C}^{0,\omega}(X)}   \omega( 2M)}{\omega(R/2)}  +  \varepsilon \leq  \frac{\|f\|_{\dot{C}^{0,\omega}(X)}   \omega( 2M)}{C_{\op{db}}^{-1}\omega(R)}  +  \varepsilon   \leq (1+C_{\op{db}}\|f\|_{\dot{C}^{0,\omega}(X)})\varepsilon.
	\]
\end{proof}

\begin{lemma}\label{lem:LipSmall} Let $X,Y$ be arbitrary normed spaces. Let $\omega$ be non-decreasing and satisfy \eqref{eq:mod:db}. Let $\tau:X\to X$ be Lipschitz. Then, 
	$$
	\VC^{0,\omega}_{\op{small}}(X,Y)\circ \tau \subset \VC^{0,\omega}_{\op{small}}(X,Y).
	$$
\end{lemma}
\begin{proof} 
	Let $\varepsilon>0$ and we need to show that if $r$ is taken sufficiently small, then
	\begin{align}
		\Norm{f\circ\tau}{\dot C^{0,\omega}_{r}(X,Y)} := \sup_{\substack{x\not=y\in X \\ \|x-y\|<r}}\frac{|(f\circ\tau)(x)-(f\circ\tau)(y)|}{\omega(\|x-y\|)} \leq \varepsilon.
	\end{align}
	Let us denote 
	$\varepsilon(r) := \Norm{f}{\dot C^{0,\omega}_{r}}$ so that  $\varepsilon(r)\to 0$ as $r\to 0.$
	Let $\|x-y\| < r$ and by $\omega$ being non-decreasing and doubling we bound
	\begin{equation*}\label{eq:tsep1}
		\begin{split}
			\frac{|(f\circ\tau)(x)-(f\circ\tau)(y)|}{\omega(\|x-y\|)} &= \frac{|f(\tau(x)) - f(\tau(y)) |}{\omega(\|\tau(x)-\tau(y)\|)} \frac{\omega(\|\tau(x)-\tau(y)\|)}{\omega(\|x-y\|)} \\ &\leq \varepsilon( \op{Lip}(\tau)r)\cdot C(C_{\op{db}},\op{Lip}(\tau)),
		\end{split}
	\end{equation*}
	and the right-hand side tends to zero as $r$ does.
\end{proof}

\begin{proof}[Proof of Theorem \ref{thm:approx:bs}]

	First, note that if $f\in \VC^{0,\omega}_{\op{small}}(X,Y)$ has bounded support, then, by property \eqref{eq:mod:coer1} of $\omega,$ $f$ belongs to $\VC^{0,\omega}(X,Y).$ Since $\VC^{0,\omega}(X,Y)$ is closed under limits with respect to the seminorm $\dot{C}^{0,\omega},$ as per Remark \ref{remark:VCareclosedsubspaces}, the inclusion $``\subset"$ holds true.

	The reverse inclusion is much more complicated. To construct the desired approximation, we define, for each $M>0,$ the following key function
	\begin{align}\label{eq:tauDEF}
		\tau_M:X \to B(0,M),\qquad	\tau_M(x) = \begin{cases}
			x,\qquad &\|x\| <M, \\
			\left(\frac{2M-\|x\|}{M}\right)^{2} x,\qquad M\leq & \|x\|<2M,\\
			0,\qquad &\|x\| \geq 2M.
		\end{cases}
	\end{align}
	 Our goal is to show that, given $f\in\VC^{0,\omega}$ and $\varepsilon >0,$ we can find $M>0$ large enough so that
	\begin{align}\label{eq:approx1}
		f\circ\tau_M\in\VC^{0,\omega}_{\op{small}},\qquad\Norm{f-f\circ\tau_M}{\dot{C}^{0,\omega}}\lesssim \varepsilon.
	\end{align}
	Notice that $f\circ\tau_M-f(0)$ is zero outside $B(0,2M),$ and hence $f \circ \tau_M$ has bounded support. (Recall that $\dot C^{0,\omega}$ is defined modulo constant equivalence classes.) The choice of $M$ will be made in terms of a small parameter $r$ and a large parameter $R,$ both depending on $\varepsilon$ and $f.$ But before doing that, we need to check some properties concerning the function $\tau_M$, for any $M>0.$

	First, we verify that
	\begin{align}\label{eq:tau1}
		\| \tau_M(x)-\tau_M(z) \|  \leq  5 \|x-z\| \quad \text{for all} \quad x,z\in X.
	\end{align}
	Indeed, for points $x,z \in B(0,2M) \setminus B(0,M),$ the definition of $\tau$ and the triangle inequality give
	\begin{equation}\label{eq:tauGrad}
		\begin{split}
			\|	 \tau_M (x) - \tau_M (z) \| &  \leq \Big  | \big(2- \frac{\|x\|}{M} \big)^{2} -  \big(2- \frac{\|z\|}{M} \big)^{2} \Big  | \|z \| +  \big(2- \frac{\|x\|}{M} \big)^{2} \|x-z\|  	 
			\\
			& \leq  \big( 2  - \frac{\|x\|}{M} +  2- \frac{\|z\|}{M} \big)  \frac{\|x-z\|}{M} \|z\|   +  \big(2- \frac{\|x\|}{M} \big)^{2} \|x-z\| \\ 
			&\leq   \big( 4  - \frac{2M}{M} \big)\frac{\|x-z\|}{M}  (2M) + \big(2- \frac{M}{M} \big)^{2} \|x-z\| = 5 \|x-z\|,
		\end{split}
	\end{equation}
	where in the second bound we used the basic identity $a^2-b^2 = (a+b)(a-b).$
	As $\tau_M$ is obviously $1$-Lipschitz in the sets $B(0,M)$ and $B(0,2M)^c,$ and $\tau_M$ is continuous in $X,$ we deduce \eqref{eq:tau1} for all $x,z\in X.$
	
	Then we show the following contraction property\footnote{ If the power $2$ in the definition \eqref{eq:tauDEF} of $\tau_M$ was replaced with the power $1,$ neither \eqref{eq:tau3} nor a variation of \eqref{eq:tau3} with a function $\theta: (0, \infty) \to (0, \infty)$ satisfying $\lim_{M\to\infty}\theta(M) = 0$ instead of $1/\sqrt{M}$ is true.}: 
%	\begin{equation}\label{eq:tau3}
%		\begin{split}
%			\Big\{ M>R \geq 1, \quad  x\in B(0,M)^c,\quad \tau_M(x),\tau_M(z)\in B(0,R)   \Big\} 
%			\Longrightarrow   \|\tau_M(x)-\tau_M(z)\|  \leq  \frac{5R}{\sqrt{M}} \|x-z\|.
%		\end{split}
%	\end{equation}
\begin{multline}\label{eq:tau3}
	\Big\{ M>R \geq 1, \quad  x\in B(0,M)^c,\quad \tau_M(x),\tau_M(z)\in B(0,R)   \Big\} \\ 
	\Longrightarrow   \|\tau_M(x)-\tau_M(z)\|  \leq  \frac{5R}{\sqrt{M}} \|x-z\|.
\end{multline}
	To show \eqref{eq:tau3}, assume $M>R\geq 1,$ and we distinguish two cases. If $\| x-z \|> M/2,$ then we have
	\begin{align*}
		\|\tau_M(x)-\tau_M(z)\| \leq 2R < 2R\frac{\|x-z\|}{M/2}  =  \frac{4R}{M}\|x-z\| \leq \frac{5R}{\sqrt{M}} \|x-z\| .
	\end{align*}  
	So next assume that $z\in B(x,M/2).$ Since $R < M$ and $\tau_M(z)\in B(0,R)$, by the definition of $\tau$ we must have necessarily $\|z\|> M.$ Thus, it remains to show that 
%	\begin{equation}\label{eq:tau4}
%		\begin{split}
%			\Big\{ M>R \geq 1, \quad  x,z\in B(0,M)^c,\quad \tau_M(x),\tau_M(z)\in B(0,R)   \Big\} 
%			\Longrightarrow  \| \tau_M(x)-\tau_M(z) \|    \leq  \frac{5R}{\sqrt{M}} \|x-z\| ,
%		\end{split}
%	\end{equation}
\begin{multline}\label{eq:tau4}
				\Big\{ M>R \geq 1, \quad  x,z\in B(0,M)^c,\quad \tau_M(x),\tau_M(z)\in B(0,R)   \Big\}  \\
	\Longrightarrow  \| \tau_M(x)-\tau_M(z) \|    \leq  \frac{5R}{\sqrt{M}} \|x-z\| ,
\end{multline}
	which is symmetric with respect to both variables $x,z.$ Notice that if $x,z\in B(0,2M)^c,$ then the left-hand side of the claimed estimate is zero. Assume that $x,z\in B(0,2M)\setminus B(0,M)$ and we make the following observation.
	By $\tau_M(x)\in B(0,R)$ and $x\in B(0,M)^c,$ there holds that 
	\[
	\Big \| \left(\frac{2M-\| x \|}{M}\right)^{2} x \Big \| \leq R, \quad  \text{and so} \quad \Babs{\frac{2M-\|  x \|}{M}}  \leq \sqrt{\frac{R}{\|x\|}}\leq \sqrt{ \frac{R}{M}} ,  
	\]
	and the same bound is valid for the variable $z.$ 
	Thus, by \eqref{eq:tauGrad}, and also using $M>R \geq 1$, we obtain
	\begin{align*}
		\| \tau_M (x) - \tau_M (z) \|  &  \leq \Big( 2  - \frac{\|x\|}{M} +  2- \frac{\|z\|}{M} \Big)  \frac{\|x-z\|}{M} \|z\|   +  \Big(2- \frac{\|x\|}{M} \Big)^{2} \|x-z\|  \\
		& \leq  \Big( 4 \sqrt{\frac{R}{M}} +   \frac{R}{M}  \Big) \|x-z\| \leq 5 \sqrt{\frac{R}{M}} \|x-z\| \leq \frac{5R}{\sqrt{M}} \|x-z\|.
	\end{align*}

	As the last case, suppose that $z\in B(0,2M)^c.$ This case follows by the continuity of $\tau_M$ and the previous case. Indeed, taking a point $y\in [x,z]$ with $\|y\|=2M,$ by the previous case applied for $x$ and $y,$ and bearing in mind that $\tau_M(z)=\tau_M(y)=0,$ we deduce
	\[
	\|\tau_M(x)-\tau_M(z)\| = \| \tau_M(x)- \tau_M(y) \| \leq  \frac{5R}{\sqrt{M}}  \| x-y \| \leq  \frac{5R}{\sqrt{M}}  \| x-z \|.
	\]
	We have shown \eqref{eq:tau4}, and thus the claim \eqref{eq:tau3}.
	
	 We now begin the verification of \eqref{eq:approx1}, so let $f \in\VC^{0,\omega}$. Since $\tau_M$ is Lipschitz \eqref{eq:tau1}, Lemma \ref{lem:LipSmall} implies the left claim on the line \eqref{eq:approx1}, for any $M>0.$ Given $\varepsilon >0,$ it remains to determine the parameter $M $ so that the right claim in \eqref{eq:approx1} holds. 
	 
	 Since $f\in  \VC^{0,\omega}_{\op{small}}$, by Lemma \ref{lem:LipSmall} (and its proof), we take $r>0$, depending on the data $f,\varepsilon,C_{\op{db}}$ (but not on $\sigma$) so that
		\begin{align}\label{eq:dens1}
			\sup_{\substack{x\not=y\in X \\ \|x-y\|<r}}\frac{|f \circ \sigma (x)- f \circ \sigma (y)|}{\omega(\|x-y\|)}\leq \varepsilon, \quad \text{for every function } \:  \sigma : X \to X \text{ with } \lip(\sigma) \leq 5.
		\end{align}
	Because $f\in  \VC^{0,\omega}_{\op{far}}$, Lemma \ref{lem:dist=far} allows us to take large $R>1$ (depending only on $\varepsilon$ and $f$) so that 
	\begin{align}\label{eq:dencc}
		\sup_{x\in B(0,R)^c}\sup_{y\in X}\frac{\abs{f(x)-f(y)}}{\omega(\| x-y \|)} \leq \vare.
	\end{align}
Finally, by the assumption \eqref{eq:mod:coer0} on $\omega$, we take $M >0$ large enough so that
		\begin{align}\label{eq:choiceofM}
			M >R, \quad  \omega(R \cdot M^{-1/4}) \leq \varepsilon \omega(r), \quad \text{and} \quad \omega(2R) \leq \varepsilon \omega(M^{1/4}).
		\end{align}
	In order to prove that \eqref{eq:approx1} holds with this \eqref{eq:choiceofM} choice of $M>0,$ we will first show that
	\begin{align}\label{eq:dens2}
		\text{for all }\,x\in B(0,M)^c \,\text{ and }\, z\in X,\,\text{ there holds that }\,\frac{\abs{(f \circ \tau_M)(x)-(f \circ \tau_M)(z)}}{\omega(\| x-z \|)}\lesssim_{\omega,f} \vare.
	\end{align}
	Fix a point $x\in B(0,M)^c$ and we distinguish cases depending on the sizes of $\tau(x)$ and $\tau(z).$ 
	
	\text{Case 1.} If $\tau_M(x)\in B(0,R)^c,$ then by \eqref{eq:dencc}, \eqref{eq:tau1}, and $\omega$ being non-decreasing and doubling \eqref{eq:mod:db},
	\begin{multline*}
	 \abs{f(\tau_M(x))-f(\tau_M(z))} 
	\leq 	\Big(\sup_{u\in B(0,R)^c}\sup_{v\in X}\frac{\abs{f(u)-f(v)}}{\omega(\| u-v\|) }\Big)\omega(\| \tau_M(x)-\tau_M(z) \|) \lesssim_{C_{\op{db}}} \varepsilon \omega(\|x-z\|).
	\end{multline*}
	
	%\smallskip
	
	\text{Case 2.} The case $\tau_M(z)\in B(0,R)^c$ is symmetrical to \text{Case 1.}

	\text{Case 3.} Now consider $\tau_M(x),\tau_M(z)\in B(0,R)$. First observe that when $\|x-z\| \geq M^{1/4},$ using that $f\in \dot{C}^{0,\omega}$ and \eqref{eq:choiceofM}, we have
	\begin{multline*}
	 \abs{f(\tau_M(x))-f(\tau_M(z))} \leq \|f\|_{\dot{C}^{0,\omega}} \omega(|\tau_M(x )-\tau_M(z )|) \\
		 \leq \|f\|_{\dot{C}^{0,\omega}} \omega(2R) \leq  \|f\|_{\dot{C}^{0,\omega}} \varepsilon  \omega(M^{1/4}) \leq \|f\|_{\dot{C}^{0,\omega}} \varepsilon \omega(\|x-z\|).
	\end{multline*}
	Therefore, as the second subcase of the current Case 3, we now assume that $\|x-z\| \leq M^{1/4}$. Moreover, \eqref{eq:dens1} and the fact that $ \tau_M $ is $5$-Lipschitz (see \eqref{eq:tau1}) imply that, to prove the estimate \eqref{eq:dens2} for $x,z,$ we may assume that $\|x-z\| \geq r$. Thus, our current standing assumptions are 
	\begin{align}\label{eq:standing}
		x \in B(0,M)^c,\quad   \tau_M(x),\tau_M(z)\in B(0,R),\quad	r \leq \|x-z\| \leq M^{1/4}.
	\end{align}
 First we bound 
	\begin{equation}\label{eq:dens3}
		\abs{f(\tau_M(x))-f(\tau_M(z))} \lesssim \Norm{f}{\dot C^{0,\omega}(B(0,R))}\omega(\| \tau_M(x)-\tau_M(z) \|) \lesssim_f \omega(\| \tau_M(x)-\tau_M(z) \|).
	\end{equation}
	Then, by the property \eqref{eq:tau3} (first two of \eqref{eq:standing} are in force), the upper bound $ \|x-z\| \leq M^{1/4}$ of \eqref{eq:standing}, then the second inequality of \eqref{eq:choiceofM}, and finally the lower bound $r \leq \|x-z\|$ of \eqref{eq:standing} give us 
	\begin{equation}\label{eq:dens33}
		\op{RHS}\eqref{eq:dens3}  \lesssim_{C_{\op{db}}} \omega\Big( \frac{R}{\sqrt{M}} \|x-z\| \Big) \leq \omega\Big( \frac{R}{\sqrt{M}} M^{1/4} \Big)  \leq \varepsilon \omega(r) \leq \varepsilon \omega(\|x-z\|),
	\end{equation}
where we used also that $\omega$ is non-decreasing.
	Chaining the estimates \eqref{eq:dens3} and \eqref{eq:dens33}, we have shown \eqref{eq:dens2} for $x,z$ as in the present \text{Case 3.}

	We are finally ready so show the right claim on the line \eqref{eq:approx1}. If $x,z\in B(0,M)$ there is nothing to prove, since $\tau(x) = x$ and $\tau(z) = z.$ The supremum over those $x\in B(0,M)^c,$ $z\in X,$ can be estimated using \eqref{eq:dencc}, the fact that $M>R$ (see \eqref{eq:choiceofM}), and \eqref{eq:dens2} as 
	\begin{align*}
		&\sup_{x\in B(0,M)^c} \sup_{z\in X} \frac{\abs{(f-(f \circ \tau_M))(x)-(f-(f \circ \tau_M))(z)}}{\omega(\|x-z\|)} \\
		&\qquad \leq 	\sup_{x\in B(0,M)^c}  \sup_{z\in X}  \frac{\abs{f(x)-f(z)}}{\omega(\|x-z\|)}+	\sup_{x\in B(0,M)^c}  \sup_{z\in X}  \frac{\abs{(f \circ \tau_M)(x)-(f \circ \tau_M)(z)}}{\omega(\|x-z\|)}\\
		& \qquad \lesssim   \sup_{x\in B(0,R)^c}  \sup_{z\in X}  \frac{\abs{f(x)-f(z)}}{\omega(\|x-z\|)}+	\vare \lesssim \vare.
	\end{align*}
	We have now shown both claims on the line \eqref{eq:approx1}, thus completing the proof of Theorem \ref{thm:approx:bs}. 
\end{proof}

We next show how to easily upgrade the above bounded support approximation into a compact and smooth approximation on $X = \R^n,$ but $Y$ is allowed to be an arbitrary Banach space. 
\begin{proof}[Proof of Theorem \ref{thm:main:Euc}] Let $f\in\VC^{0,\omega}(\R^n,Y)$ and by Theorem \ref{thm:approx:bs} find $g\in\VC_{\op{small}}^{0,\omega}(\R^n,Y)$ with bounded support such that $\| f-g\|_{\dot C^{0,\omega}(\R^n,Y)}<\varepsilon.$ Let $r>0$ and $0\leq \eta_r\in C^{\infty}_c(B(0,r))$ be a standard real-valued smooth bump in $\R^n$ with $\int_{\R^n}\eta_r = 1,$ and we define
$$
h_r(x)=g*\eta_r (x) = \int_{\R^n} g(x-z) \eta_r(z) \ud z = \int_{B(0,r)} g(x-z) \eta_r(z) \ud z, \quad x\in \R^n ,
$$
with the integral in the Bochner sense. 

For the sake of completeness, let us justify why the integral $h_r(x)$ is well-defined for every $x\in \R^n.$ Denote momentarily $B:=B(0,r).$ The function
$$
B \ni z \mapsto \sigma_x(z):=g  ( x-z  ) \eta_r(z) 
$$
is uniformly continuous $B \to Y$, by the compactness of $B$ and the continuity of the functions $g$ and $\eta_r.$ In particular, for each $y^* \in Y^*,$ one has that the mapping $B \to \R$ given by $B \ni z \mapsto y^*( \sigma_x(z))$ is continuous, and thus Lebesgue measurable over $B.$ Also, the continuity of $\sigma_x$ implies that the image $\sigma_x(B)$ is compact, and so the closed linear span
$$
\overline{\mathrm{span}}\lbrace \sigma_x(z) \, : \, z\in B \rbrace 
$$
is a separable subspace of $Y.$ By \cite[Proposition 5.1]{BenyaminiLindenstrauss}, these two properties guarantee the measurability of $\sigma_x$ in $B.$ Then, by \cite[Proposition 5.2]{BenyaminiLindenstrauss} the Bochner-integrability of $\sigma_x$ is equivalent to $\| \sigma_x \|_Y \in L^1(B),$ that is, $\int_B \| \sigma_x(z) \|_Y \ud z < \infty.$ But this is an immediate consequence of the fact that $\eta_r$ is a test function, and that $g,$ being a function of $\dot{C}^{0,\omega}(\R^n,Y)$, is bounded on bounded sets. 

Notice that $h_r$ is compactly supported, as are $g$ and $\eta_r.$ 
Moreover, since $\eta_r\in C^{\infty}(\R^n,Y)$ is with compact support, writing
$$
h_r(x) =  \int_{\R^n} g(z) \eta_r(x-z) \ud z,
$$
and using Differentiation Under the Integral Sign theorems for the Bochner integral (see \cite[Corollary 91]{HajJoh14}), one obtains that $h_r \in C^\infty(\R^n,Y)$ as well.

	Provided that $|u|\leq \delta,$ there holds uniformly in $r$ that 
	\[
	|h_r(x+u)-h_r(x)| = \Babs{\int_{\R^n} \eta_r(z)(g(x+u-z)-g(x-z))\ud  z }\leq \Norm{g}{\dot C^{0,\omega}_{\delta}}\omega(|u|),
	\]
	and hence we find $\delta>0$ so that, uniformly in $r,$ we have 
	$\| g-h_r\|_{\dot C^{0,\omega}_{\delta}} \leq \| g\|_{\dot C^{0,\omega}_{\delta}}  + \| h_r\|_{\dot C^{0,\omega}_{\delta}} \leq  \varepsilon.$ To deal with the scales $\geq \delta$ we argue as follows.
	Observe that $g\in\VC^{0,\omega}_{\op{small}}$ is uniformly continuous, by $\lim_{t\to 0}\omega(t)  = 0,$ and thus the approximation $h_r := g*\eta_r$ converges uniformly to $g,$ as $r\to 0.$ We let $r = r(\delta)$ be so small that $\| g-h_r\|_{\infty} \leq \frac{1}{2}\varepsilon\omega(\delta).$ Then, for $|x-z|>\delta,$ there holds that 
	\[
	\| (g-h_r)(x)-(g-h_r)(z)\| \leq 2\| g-h_r\|_{\infty} \leq \varepsilon\omega(\delta) \leq \varepsilon\omega(|x-z|),
	\]
	by  $\omega$ being non-decreasing.
	This concludes the proof of 
	$\VC^{0,\omega}(\R^n,Y)  \subset  \overline{C^{\infty}_c(\R^n,Y)}^{\dot C^{0,\omega}(\R^n,Y)}.$ 
	
	For the reverse inclusion $``\supset ",$ simply notice that by the Mean Value Inequality, $C^{\infty}_c(\R^n,Y)\subset \op{Lip}_{\op{bs}}(\R^n,Y)$. We remind that this inequality holds for smooth $Y$-valued functions, as a consequence of the Hahn-Banach theorem on $Y.$ Also, by \eqref{eq:mod:coer1} there holds that $ \op{Lip}_{\op{bs}}(\R^n,Y)\subset \VC^{0,\omega}(\R^n,Y).$ As $\VC^{0,\omega}(\R^n,Y)$ is closed with respect to $\dot C^{0,\omega}$ limits, by Remark \ref{remark:VCareclosedsubspaces}, we are done.
\end{proof}

Next, we record a quantitative version of Theorem \ref{thm:main:Euc} to be used in the proof of Theorem \ref{thm:MarOik24} in \cite{MarOik24}. While Corollary \ref{thm:quant} below is one the key elements in applications to the bi-commutator in \cite{MarOik24},  it plays no further role in the present article. 
\begin{dfn}\label{def:VCalt} For a subset $G$ of $\VC^{0,\omega}(\R^n,Y),$ we write  $
	G \subset_u \VC^{0,\omega}(\R^n,Y)$ (with $``u"$ for uniformly) provided that $G$ is equibounded in the $\dot{C}^{0,\omega}$ seminorm, and that 
	for all $\varepsilon>0,$ there exists $t>0$ such that if $|x-y|<t$ or $|x|>t^{-1}$ or $|y|>t^{-1},$ then 
	\[
	\sup_{g\in G} \frac{|g(x)-g(y)|}{\omega(|x-y|)}< \varepsilon.
	\]
\end{dfn}

By an inspection of the above proofs, the reader convinces themselves that the following is true. 
\begin{corollary}\label{thm:quant} Let $Y$ be a Banach space and the modulus $\omega$ satisfy  \eqref{eq:mod:coer0} and \eqref{eq:mod:db}. 
	Let $G\subset_u 	\VC^{0,\omega}(\R^n,Y).$ Then,
	\begin{itemize}
		\item for all $L>0$ we have $G\circ\tau_L\subset_u 	\VC^{0,\omega}(\R^n,Y),$
		\item for all $L,L'>0$ we have $(G\circ\tau_L)*\eta_{L'}\subset_u 	\VC^{0,\omega}(\R^n, Y),$
		\item for all $M=M(\varepsilon)>0$ sufficiently large 
		\begin{align}\label{eq:Lem:MO1}
			\sup_{f\in G} \|  f - f\circ\tau_M  \|_{\dot{ \op{C}}^{0,\omega}(\R^n, Y)}  \leq \varepsilon,
		\end{align}
		\item  for all $K = K(M,\varepsilon)>0$ sufficiently large
		\begin{align}\label{eq:Lem:MO2}
			\sup_{f\in G}\|  f\circ\tau_{M} -( f\circ\tau_M)*\eta_{1/K}  
			\|_{\dot{ \op{C}}^{0,\omega}(\R^n,Y)} \leq \varepsilon.
		\end{align}
	\end{itemize}
\end{corollary}

Proving approximations with a more general normed space $X$ and $Y=\R$ (or $Y=\C$) is a delicate task and this is the content of the next Section \ref{sect:approx:inftyDIM}.
But before that, we finish this section by proving the connection between pointwise and mean oscillation type conditions.

\begin{proof}[Proof of Theorem \ref{thm:cubepw}] We first prove the norms equivalent. Fix an arbitrary cube $Q$ and estimate 
	\begin{align*}
		\fint_Q | f(x)-\ave{f}_Q|_Y \ud x \leq \fint_Q\fint_Q | f(x)-f(y)|_Y\ud x\ud y \leq 	\Norm{f}{\dot C^{0,\omega}(\R^n,Y)} \fint_Q\fint_Q \omega(\aabs{x-y}) \ud x\ud y 
	\end{align*}
	and we continue the bound with 
	\begin{align}\label{eq:step:dini}
		\fint_Q\fint_Q \omega(\aabs{x-y}) \ud x\ud y\lesssim_{n,C_{\op{db}}}\omega(\ell(Q)),
	\end{align}
	thus we obtain $\Norm{f}{\dot C^{0,\omega}(\R^n,Y)}\gtrsim_{n,C_{\op{db}}}\Norm{f}{\BMO^\omega(\R^n,Y)}.$ 
	
	For the other direction, fix $x,z\in\R^n$ and let $Q_0$ be a cube containing both $x$ and $y$ such that $\ell(Q_0) = |x-y|.$ Consider the dyadic descendants of $Q_0$ (achieved by iteratively halving the sides) and for every integer $k \geq 0$ let $Q_k(x)$ be the descendant of $Q_0$ of sidelength $2^{-k}\ell(Q_0)$ that contains the point $x.$ Similarly, we define $Q_k(y)$ for each $k.$ By the continuity of $f:\R^n\to Y,$ we may write
	\begin{align}\label{eq:scope1}
		f(x)-f(y) = \Big(\sum_{k=0}^{\infty}\ave{f}_{Q_{k+1}(x)} - \ave{f}_{Q_{k}(x)} \Big)-  \Big(\sum_{k=0}^{\infty}\ave{f}_{Q_{k+1}(y)} - \ave{f}_{Q_{k}(y)}\Big).
	\end{align}
	Both sums are estimated identically, so it is enough to consider the first. We bound
	\begin{align*}
		\Babs{  \sum_{k=0}^{\infty}\ave{f}_{Q_{k+1}(x)} - \ave{f}_{Q_{k}(x)}}_Y &\lesssim_n \sum_{k=0}^{\infty} \fint_{Q_{k}(x)}|f-\ave{f}_{Q_ k(x)}|_Y \leq \Norm{f}{\BMO^{\omega}(\R^n,Y)}\sum_{k=0}^{\infty}\omega(\ell(Q_k(x))) 
	\end{align*}
	and we continue, using that $\omega$ is non-decreasing and the condition \eqref{eq:mod:dini}, with
	\begin{align*}
		\sum_{k=0}^{\infty}\omega(\ell(Q_k(x)))  = \sum_{k=0}^{\infty} \ell(Q_k(x)) \frac{\omega(\ell(Q_k(x)))}{\ell(Q_k(x))} \leq  2\int_0^{\ell(Q_0)}\frac{\omega(t)}{t} \ud t \lesssim [\omega]_{*}  \omega(\ell(Q_0)).
	\end{align*}
	Since $\omega(\ell(Q_0))= \omega(|x-y|),$ it follows that  $\Norm{f}{\dot C^{0,\omega}(\R^n,Y)}\lesssim_{n} [\omega]_{*}  \Norm{f}{\BMO^\omega(\R^n,Y)}.$

		It is clear that the above argument gives $\VC^{0,\omega}_{\Gamma}(\R^n,Y) = \VMO^{\omega}_{\Gamma}(\R^n,Y)$, when $\Gamma=\op{small},$ as well as the inclusion $\VC^{0,\omega}_{\Gamma}(\R^n,Y) \subset \VMO^{\omega}_{\Gamma}(\R^n,Y)$, when $\Gamma=\op{far}.$ 
		The reverse inclusion $``\supset"$ for $\Gamma=\op{far}$ can be seen as follows.  Notice first that the assumption $f\in \VMO^{\omega}_{\op{far}}(\R^n,Y),$ by the above proof, gives the following. For all $\varepsilon>0$ there exists some $M = M(\varepsilon) > 0$ so that 
		\begin{align}\label{eq:jug}
			\sup_{\dist(Q,0) \geq M} \sup_{x,y\in Q} \frac{|f(x)-f(y)|}{\omega(|x-y|)} \leq \varepsilon.
		\end{align}
		If the line $[x,y]$ connecting $x$ to $y$ satisfies $\dist([x,y],0) \geq M,$ then both $x,y$ can be enclosed in some large cube $Q$ such that $\dist(Q,0)\geq M.$ Therefore, by \eqref{eq:jug}, this case is clear.
		
		Let $M = M(\varepsilon)$ be such that \eqref{eq:jug} holds and consider $x,y$ such that $|x|,|y|> R$ but $\dist([x,y],0) < M$ and $R\gg M.$ We consider $M$ fixed and will show that taking $R = R(M,\|f\|_{\dot{C}^{0,\omega}}, \omega)$  sufficiently large yields a correct bound.
		Since $R\gg M$, we find points $z_1,z_2\in [x,y]$ such that $\dist([x,z_1],0) \geq M$, $\dist([z_2,y],0) \geq M$ and $|z_1|,|z_2| \leq M.$ Then, using \eqref{eq:jug} for the pairs of points $(x,z_1)$ and $(z_2,y)$ we bound 
		\begin{multline*}
			|f(x)-f(y)| \leq |f(x)-f(z_1)| + |f(z_1)-f(z_2)| + |f(z_2)-f(y)|  \\
			\leq \varepsilon\omega(|x-z_1|) +  \|f\|_{\dot{C}^{0,\omega}}\omega(|z_1-z_2|) + \varepsilon\omega(|z_2-y|)\leq \Big(2\varepsilon + \frac{\omega(|z_1-z_2|)}{\omega(|x-y|)}\|f\|_{\dot{C}^{0,\omega}}\Big)\omega(|x-y|).
		\end{multline*}
		To get a bound of a multiple of an $\varepsilon$ for the bracketed factor, we use that  $|z_1-z_2| \leq |z_1| + |z_2| \leq 2M$ and that $|x-y| \geq R/2$. Indeed, for the latter bound, note that if $|x-y| < R/2$, then by $|x|,|y|> R$, for each $z \in [x,y]$, we can bound $|z| \geq |x| - |x-y| \geq R/2 > M$ which contradicts $\dist([x,y],0) \leq M.$
		Using that $\omega$ is non-decreasing and $\omega(\infty) = \infty$ we take $R = R(M,\|f\|_{\dot{C}^{0,\omega}}, \omega)$ so that
		\[
		\frac{\omega(|z_1-z_2|)}{\omega(|x-y|)}\|f\|_{\dot{C}^{0,\omega}} \leq  \frac{\omega(4M)}{\omega(R)}\|f\|_{\dot{C}^{0,\omega}} \leq \varepsilon.
		\]
		This concludes the proof of the inclusion
		$\VC^{0,\omega}_{\op{far}}(\R^n,Y) \subset \VMO^{\omega}_{\op{far}}(\R^n,Y).$
	
	For $\Gamma = \op{large},$ a further argument is needed for both inclusions. We begin by showing that $\VC^{0,\omega}_{\op{large}}(\R^n,Y) \supset  \VMO^{\omega}_{\op{large}}(\R^n,Y).$ Let $f\in \VMO^{\omega}_{\op{large}}(\R^n,Y), $ and $\varepsilon>0,$ and let $K \in \N$ be such that if $\ell(Q)\geq 2^K,$ then $\calO^{\omega}(f;Q)\leq \varepsilon.$  Let $N \gg K$ be a large integer to be specified later, and suppose $x,y\in \R^n$ are so that $|x-y| \geq 2^{N}.$ Let $M \geq N$ be so that $ 2^M \leq \abs{x-y}\leq  2^{M+1} ,$ and let $Q_0$ be a cube containing both $x,y$ and with $\ell(Q_0) = |x-y|.$
	Then, again considering the left sum in the expansion \eqref{eq:scope1}, we bound
	\begin{align*}
		&\Babs{\sum_{k=0}^{\infty}\ave{f}_{Q_{k+1}(x)}-\ave{f}_{Q_k(x)}}_Y \leq  \Big(\sum_{k=0}^{M-K} + \sum_{k=M-K+1}^{\infty}\Big) |\ave{f}_{Q_{k+1}(x)}-\ave{f}_{Q_k(x)} |_Y \\
		&\lesssim_n \varepsilon\sum_{k=0}^{M-K}\omega(\ell(Q_k(x))) +  \sup_{\ell(Q)\leq 2^K}\calO^{\omega}(f;Q) \sum_{k=M-K+1}^{\infty}  \omega(\ell(Q_k(x))) \\
		&\lesssim_{C_{\op{db}}}  \varepsilon\int_{2^{K}}^{2^{M+1}}\frac{\omega(t)}{t} \ud t+   \sup_{\ell(Q)\leq 2^{K}}\calO^{\omega}(f;Q)\int_{0}^{2^{K}}\frac{\omega(t)}{t} \ud t\\
		&\lesssim_{C_{\op{db}}}
		\Big(  \varepsilon \Big(\frac{1}{\omega(2^{M+1})} \int_{2^{K}}^{2^{M+1}}\frac{\omega(t)}{t} \ud t\Big) \\
		&\qquad\qquad\qquad +   \frac{\omega(2^K)}{\omega(2^{M+1})}\sup_{\ell(Q)\leq 2^K}\calO^{\omega}(f;Q)\Big(\frac{1}{\omega(2^{K})}\int_{0}^{2^{K}}\frac{\omega(t)}{t} \ud t\Big)\Big)\omega(\ell(Q_0)) \\
		&   \lesssim \Big(  \varepsilon [\omega]_{*} +   \frac{\omega(2^K)}{\omega(2^{N+1})}\Norm{f}{\BMO^{\omega}(\R^n,Y)}[\omega]_{*} \Big)\omega(|x-y|) . 
	\end{align*}
	By the condition $\omega(\infty) = \infty$, we choose $N$ large enough and bound the bracketed factor from above by $\leq 2 \varepsilon[\omega]_{*} .$ Repeating the same proof with $y$ in place of $x$, we control the right sum on the line \eqref{eq:scope1} and thus obtain 
	$
	|f(x)-f(y)| \leq 4 \varepsilon[\omega]_{*}  \omega(|x-y|),
	$
	that is, $f\in \VC^{0,\omega}_{\op{large}}(\R^n,Y) .$

	Then, we show that $\VC^{0,\omega}_{\op{large}}(\R^n,Y) \subset  \VMO^{\omega}_{\op{large}}(\R^n,Y) .$ Let $f\in \VC^{0,\omega}_{\op{large}}(\R^n,Y)$, $\varepsilon>0,$ and $R>0$ be such that $\osc^{\omega}_{\delta}(f)\leq \varepsilon,$ if $\delta\geq R.$ Let $M \gg R $ be a large constant, which we will specify later, and $Q\in\calQ$ a cube such that $\ell(Q) \geq M.$ For every $y\in Q ,$ denote by $Q_R(y)$ the cube centered at $y$ and of $\diam(Q_R(y)) = R.$ Then, we have 
	\begin{align*}
		&\fint_{Q}	\int_{Q}\aabs{f(x)-f(y)}_Y\ud x\ud y = \fint_{Q}	\Big(\int_{Q\setminus Q_R(y)}+ \int_{Q \cap Q_R(y)}\Big)\aabs{f(x)-f(y)}_Y\ud x\ud y \\ 
		&\leq \varepsilon \fint_{Q}\int_{Q\setminus Q_R(y)} \omega(\aabs{x-y})\ud x\ud y +  \sup_{\delta \leq R}\osc^{\omega}_{\delta}(f)\fint_{Q}	 \int_{Q\cap Q_R(y)}\omega(\aabs{x-y})\ud x\ud y \\ 
		&\leq \Big( \varepsilon  \fint_Q\fint_Q \omega(\aabs{x-y}) \ud x\ud y  + \sup_{\delta \leq R}\osc^{\omega}_{\delta}(f) \fint_{Q}	\frac{1}{|Q|} \int_{ Q_R(y)}\omega(\aabs{x-y})\ud x\ud y  \Big)|Q|.
	\end{align*}
	By repeating the bound \eqref{eq:step:dini}, the left term in the brackets is bounded from above by $\lesssim_{n, C_{\op{db}}}\varepsilon\omega(\ell(Q)),$ which is a bound of the correct form.  For the right term we bound
	\[
	\lesssim_n \| f\|_{\dot C^{0,\omega}(\R^n,Y)}	\left( \frac{R}{\ell(Q)} \right)^n \omega(R) \leq \| f\|_{\dot C^{0,\omega}(\R^n,Y)} \omega(R) \leq \varepsilon \omega(\ell(Q));  	\]
	as $\omega(\infty) = \infty,$ and $M$ is chosen sufficiently large.
\end{proof}

\medskip

\section{Approximation in infinite dimensional spaces}\label{sect:approx:inftyDIM}

In this section we study approximations in Banach spaces with respect to the $\dot{C}^{0,\omega}$ seminorm. In Subsections \ref{subsect:LipSmoApprox} and \ref{subsect:SupRefApprox} we establish results for real-valued functions $f: X \to \R$, where the approximating functions are not only $C^k$ or $C^{1,\alpha}$ smooth but also Lipschitz and with bounded support.  In the last Subsection \ref{subsect:C0GamApprox} we obtain $C^\infty$ smooth approximations of Banach-valued mappings $f:c_0(\mathcal{A}) \to Y,$ that are not necessarily Lipschitz, but have bounded support and belong to $\VC^{0,\omega}(X,Y).$ As a corollary, we deduce $C^\infty$ and  $\VC^{0,\omega}$ and approximations with bounded support for all $\VC^{0,\omega}$ functions defined on $\R^n$ and with values in any Banach space.

As we mentioned in the introduction, for general background on smooth analysis on Banach spaces, including smooth renomings and approximations, we refer the reader to the monographs \cite{BenyaminiLindenstrauss} by Benyamini and Lindenstrauss; \cite{DevilleGodefroyZizler} by Deville, Godefroy, and Zizler; and \cite{HajJoh10} by H\'{a}jek and Johanis.

\subsubsection*{Assumptions on the modulus}
Throughout this whole Section \ref{sect:approx:inftyDIM} we assume that the modulus satisfies all the assumptions of Theorem \ref{thm:approx:bs}, i.e. $\omega$ is non-decreasing and satisfies \eqref{eq:mod:coer0}, \eqref{eq:mod:coer1} and \eqref{eq:mod:db}.

\subsubsection*{Notation and basic definitions} 
\begin{itemize}

\item By $\lip(X)$ we denote the class of real-valued Lipschitz functions (not necessarily bounded) on $X.$ We say that $g: X \to \R$ is $L$-Lipschitz provided that $|g(x)-g(y)| \leq L \|x-y\|$ for every $x,y\in X,$ and we denote the minimal Lipschitz constant of $g$ by $\lip(g).$ 

\item We denote by $X^*$ the (continuous) dual of $X,$ and the dual norm in $X^*$ by $\| \cdot \|_*,$ that is,
$$
\| \Lambda \|_*:= \sup \lbrace | \Lambda(v) | \, : \, \|v\| \leq 1 \rbrace,
$$
for every $\Lambda \in X^*.$ 

\item When speaking of differentiability of functions or mappings $f : X \to Y$, we always mean differentiability in the Fr\'{e}chet sense. We say, of course, that $f$ is of class $C^k(X,Y),$ when $f$ has Fr\'{e}chet derivatives up to order $k,$ and those derivatives are continuous on $X.$  

\item We say that a norm $\| \cdot \|$ on $X$ is \textit{of class} $C^k$ if $\| \cdot \| \in C^k(X \setminus \lbrace 0 \rbrace),$ for $k \in \N \cup \lbrace \infty \rbrace.$ 

\item For any class of functions $\mathcal{C}(X,Y),$ the subscripted set $\mathcal{C}_{\op{bs}}(X,Y)$ consists of functions $h: X \to Y$ of class $\mathcal{C}$ that have bounded support, meaning that there exists $R>0$ so that $h(x)=0$ for every $x\in X\setminus B(0,R).$

\item Finally, a \textit{bump} function $h: X \to Y$ is a non-zero function with bounded support.  

\end{itemize}

\subsection{Lipschitz and smooth approximations}\label{subsect:LipSmoApprox} In a real normed space $X$, we consider real-valued functions $``f"$ of the class $\VC^{0,\omega}(X, \R)$, the last abbreviated by $\VC^{0,\omega}(X).$ By Theorem \ref{thm:approx:bs} these functions $``f"$ can be approximated by functions, say $``g",$ of class $ \VC^{0,\omega}_{\op{small}}(X)$ with bounded support, in the $\dot{C}^{0,\omega}(X)$ seminorm. Our goal is to further approximate these functions $``g"$ by Lipschitz functions (whose regularity will vary depending on the smoothness properties of $X$) with bounded support. In approximating $``g"$ a crucial step is to reduce the problem to the problem of approximating Lipschitz functions by smooth Lipschitz functions with good control on the Lipschitz constants. The main results of this subsection are Theorem \ref{thm:main:LipApprox}, valid for arbitrary normed spaces, and Theorem \ref{thm:main:ApproxSeparable}, valid for separable normed spaces $X$ with a fixed degree of smoothness. The proofs of these two Theorems \ref{thm:main:LipApprox} and \ref{thm:main:ApproxSeparable} employ several technical lemma, which will be reused in the later Subsections \ref{subsect:SupRefApprox} and \ref{subsect:C0GamApprox}.

We begin by proving these technical lemmas and for this purpose, it will be useful to consider the \textit{minimal modulus of continuity} of a function $f: X \to \R,$ i.e.,
$$
\omega_f(t):= \sup \lbrace |f(x)-f(y)| \, : \, \|x-y\| \leq t \rbrace, \quad t >0.
$$
Note that $\omega_f: [0, \infty] \to [0, \infty]$ is non-decreasing, and if $f$ is uniformly continuous, then $\lim_{t \to 0} \omega_f(t)=0.$ Moreover, if $f\in \dot{C}^{0,\omega}(X),$ then $\omega_f \leq \| f\|_{\dot{C}^{0,\omega}(X)}\omega,$ and especially $\omega_f(t)$ is finite at every $t,$ but, in general, $\omega$ and $\omega_f$ are incomparable. Furthermore, since $X$ is a normed space, it is easy to see that $\omega_f$ is sub-additive, i.e. $\omega_f(s+t) \leq \omega_f(t) + \omega_f(s)$ for all $t,s>0.$ This implies
\begin{equation}\label{eq:basicpropertyminimalomega_f}
\omega_f( t) \leq 2 t \,\omega_f(1), \quad \text{for all} \quad t \geq 1.
\end{equation}

\begin{lemma}\label{lem:VCsLipApprox} Let $X$ be a normed space and $f\in \VC^{0,\omega}_{\op{small}}(X).$ Then, there exists a sequence $(f_n)_n\subset\op{Lip}(X)$ of functions converging to $f$ in the $\dot{C}^{0,\omega}(X)$-seminorm. Moreover, if $f$ has bounded support, the sequence can be taken  $(f_n)_n\subset\op{Lip}_{\op{bs}}(X)$ to have bounded supports. 
\end{lemma}
\begin{proof}
As $f\in \VC^{0,\omega}_{\op{small}}(X),$ by the comments preceding the present lemma, the minimal modulus of continuity $\omega_f$ is sub-additive and satisfies $\omega_f(t)\to 0$, as $t\to 0;$  then, it is a known result that the sequence
	$$
	f_n(x)=\inf_{y\in X} \lbrace f(y)+ n \|x-y\| \rbrace, \quad x\in X,
	$$
	converges uniformly to $f$ in $X$ and each $f_n$ is $n$-Lipschitz; for a proof see e.g. \cite[p. 408]{HajJoh14}.

	Now, let us see that $\|f-f_n\|_{\dot{C}^{0,\omega}(X)} \to 0,$ as $n\to\infty.$ Indeed, given $\varepsilon>0$ there exists $\delta>0$ so that $\|x-z \| \leq \delta$ implies $|f(x)-f(z)| \leq \varepsilon \omega( \|x-z\|).$ Let us take $N \in \N$ large enough so that $\sup_{X}|f-f_n| \leq \varepsilon \omega(\delta)$ for each $n \geq N.$ Using that $\omega$ is non-decreasing, this gives the estimate
	\begin{align*}
		|(f-f_n)(x)-(f-f_n)(z)| \leq 2\|f-f_n\|_\infty \leq \varepsilon \omega(\delta) &\leq \varepsilon \omega(\|x-z\|),  \\  
		&\text{whenever} \quad \|x-z\| \geq \delta, \, n \geq N.
	\end{align*}
	Now, for fixed $x,z\in X$ such that $\|x-z\| \leq \delta,$ we proceed as follows. Given $\eta>0,$ let $y=y(x,n,f,\eta)$ be so that
	$$
	f_n(x) \geq f(y)+n\|x-y\| - \eta.
	$$
	This choice of $y$ and the definition of $f_n$ yield
	\begin{align*}
		(f-f_n)(x)  & -(f-f_n)(z)  \leq f(x)-f(y)-n\|x-y\| + \eta-f(z)+f_n(z) \\
		& \leq f(x)-f(y)-n\|x-y\|+ \eta-f(z) + f(y+z-x)+n\|z-(y+z-x)\| \\
		& \leq f(x)-f(z)-(f(y)-f(y+z-x)) + \eta \leq \varepsilon \omega(\|x-z\|) + \varepsilon\omega(\|x-z\|) + \eta.
	\end{align*}
	Letting $\eta \to 0,$ we get
	$$
	(f-f_n)(x)    -(f-f_n)(z)  \leq 2\varepsilon\omega(\|x-z\|).
	$$
	The same argument swapping $x$ and $z,$gives $|(f-f_n)(x)    -(f-f_n)(z) | \leq 2\varepsilon\omega(\|x-z\|),$ for every $n\in \N.$ We conclude that $\|f-f_n\|_{\dot{C}^{0,\omega}(X)} \leq 2\varepsilon,$ for $n \geq N.$

	For the assertion concerning the boundedness of the supports, we may assume that $f(y)=0,$ whenever $\|y\|\geq R,$ for some $R>0.$ We will localize the infimum defining $f_n.$ For each $y\in X,$ it follows from the definitions of $f_n$ and $\omega_f$ that
	\begin{align}\label{eq:bs1}
	f(y)+n\|x-y\| &\geq f(x)-  \omega_f(\|x-y\|) + n \|x-y\|  \nonumber \\ 
			& \geq f_n(x) -  \omega_f(\|x-y\|) + n \|x-y\|.
	\end{align}
Using property \eqref{eq:basicpropertyminimalomega_f} of the minimal modulus of continuity $\omega_f$, and the observations subsequent to the definition of $\omega_f,$ we have that
$$
\omega_f(\|x-y\|) \leq 2 \|x-y\| \omega_f(1) \leq 2 \omega(1) \|f\|_{\dot{C}^{0,\omega}(X)} \|x-y\|, \quad \text{when} \quad \|x-y\| \geq 1.
$$
Therefore, for $n > 2 \omega(1) \|f\|_{\dot{C}^{0,\omega}(X)} $, one has that $ \op{RHS\eqref{eq:bs1}}>f_n(x),$ provided that $\| x-y\|\geq 1.$ These observations show that the infimum defining $f_n(x)$ is restricted to $B(x,1):$  
		$$
		f_n(x)=\inf_{y\in B(x,1)} \lbrace f(y)+ n \|x-y\| \rbrace, \quad x\in X.
		$$
		Now, if $\|x\| \geq R+ 1, $ then $\|y\| \geq R $, and so $f(y)=0,$ whenever $y\in B(x,1).$ Thus, by the previous formula, we deduce
		$
		f_n(x)= \inf_{y\in B(x,1)} \lbrace   n \|x-y\| \rbrace=0,
		$
		which shows that $f_n$ has bounded support, for $n$ large enough.

\end{proof}

In the following lemma we show that a uniform approximation by Lipschitz functions with controlled Lipschitz constants yields a $\dot{C}^{0,\omega}(X)$ approximation.

\begin{lemma}\label{lem:Lip-Comeg:Approx}
	Let $f: X \to \R$ and $(f_n)_n\subset \op{Lip}(X)$ be a sequence such that $\limsup_n \mathrm{Lip}(f_n)< \infty$ and $f_n$ converges to $f$ uniformly on $X.$ Then, $ \limsup_n \|f_n-f\|_{\dot{C}^{0,\omega}(X)}=0.$   
\end{lemma} 
\begin{proof}
	By passing to a subsequence, we may assume that $L:=\sup_{n} \mathrm{Lip}(f_n)< \infty$. Then, $f$ is Lipschitz on $X,$ with $\lip(f) \leq L.$ Given $\varepsilon>0,$ by $\lim_{t\to 0}t/\omega(t)=0,$ we find $\delta>0$ so that
	$$
	\frac{t}{\omega(t)} \leq \frac{\varepsilon}{1+L}, \quad \text{whenever} \quad t \leq \delta.
	$$
	Let $N\in \N$ be such that $\sup_{X} |f_n-f| \leq \varepsilon \omega(\delta)$ for all $n\geq N.$ Because $f$ and each $f_n$ is $L$-Lipschitz, for any two distinct points $x,z\in X$ with $\|x-z\| \leq \delta,$ we have 
	$$
	\max\left\lbrace \frac{|f(x)-f(z)|}{\omega(\|x-z\|)}, \frac{|f_n(x)-f_n(z)|}{\omega(\|x-z\|)} \right\rbrace \leq  L \frac{\|x-z\|}{\omega(\|x-z\|)} \leq \varepsilon;
	$$
	then, triangle inequality gives
	$$
	\frac{|(f-f_n)(x)-(f-f_n)(z)|}{\omega(\|x-z\|)}  \leq \frac{|f(x)-f(z)|}{\omega(\|x-z\|)}  + \frac{| f_n (x)- f_n (z)|}{\omega(\|x-z\|)}  \leq 2 \varepsilon.
	$$
	And for scales $\|x-z\| \geq \delta,$ observe that for $n \geq N,$ using that $\omega$ is non-decreasing, we have 
	$$
	|(f-f_n)(x)-(f-f_n)(z)| \leq 2 \sup_{X} |f-f_n|  \leq 2 \varepsilon \omega(\delta) \leq 2 \varepsilon \omega(\|x-z\|).
	$$
\end{proof}

We are now ready to prove our first main Theorem \ref{thm:main:LipApprox} on Lipschitz approximations. It does not involve smoothness, but holds in any normed space and will be a key ingredient in the proofs of Theorems \ref{ThmApproxsuper-reflexive} and \ref{thm:main:c0YApprox} of later sections.

\begin{theorem}[Lipschitz approximation of $\VC^{0,\omega}$]\label{thm:main:LipApprox}
Let $X$ be a normed space. Then, the following hold:
	\begin{itemize}
		\item[(i)] $\overline{\mathrm{Lip}(X) \cap \dot{C}^{0,\omega}(X) }^{\dot{C}^{0,\omega}(X)}= \VC^{0,\omega}_{\op{small}}(X),$ 
		\item[(ii)] $\overline{\mathrm{Lip}_{\op{bs}}(X)}^{\dot{C}^{0,\omega}(X)}=\VC^{0,\omega}(X). $
	\end{itemize}
	
\end{theorem}
\begin{proof} \item[(i)] The inclusion $``\subset"$ follows immediately from Lemma \ref{lem:VCsLipApprox}. For the converse inclusion $``\supset",$ we observe that condition \eqref{eq:mod:coer1} gives $ \lip(X) \cap \dot{C}^{0,\omega}(X) \subset \VC^{0,\omega}_{\op{small}}(X)$, and then recall that $\VC^{0,\omega}_{\op{small}}(X)$ is closed.

	\item[(ii)] By Theorem \ref{thm:approx:bs} we know that if $f\in \VC^{0,\omega}(X),$ and $\varepsilon>0,$ we can find $g\in \VC^{0,\omega}(X)$ with bounded support and $\|f-g\|_{\dot{C}^{0,\omega}(X)} \leq \varepsilon.$ Then, Lemma \ref{lem:VCsLipApprox} provides us with $h\in \lip_{\op{bs}}(X)$ so that $\|h-g\|_{\dot{C}^{0,\omega}(X)} \leq \varepsilon,$ and therefore $\|f-h\| _{\dot{C}^{0,\omega}(X)} \leq 2 \varepsilon;$ this shows the inclusion $``\supset".$

For the converse implication $``\subset"$, it is enough to observe that  $\mathrm{Lip}_{\op{bs}}(X)\subset \VC^{0,\omega}(X)$
		and that $\VC^{0,\omega}(X) $ is a closed subspace of $\dot{C}^{0,\omega}(X);$ see Remark \ref{remark:VCareclosedsubspaces}.
\end{proof}

Performing smooth approximations in the appropriate Banach spaces will require more work. We first show that if a normed space $X$ has a  $C^k$ and Lipschitz bump function, from now on a $C^k\cap\op{Lip}$ bump, then a $C^k\cap\op{Lip}$ approximation with uniformly bounded Lipschitz constants, of a Lipschitz function $g$ with bounded support, can be upgraded to a $C^k\cap\op{Lip}$ approximation with uniformly bounded Lipschitz constants \emph{and bounded supports}.

\begin{lemma}\label{lem:SmoothLipBsApprox}
	Let $k \in \N \cup \lbrace \infty \rbrace$ be fixed and let $X$ be a normed space that admits a $C^k\cap \op{Lip}$ bump. Let $g\in \op{Lip}(X)$ and $(g_n)_n\subset C^k(X)\cap\op{Lip}(X)$ be a sequence of functions converging uniformly to $g$ on $X,$ and with $\limsup_n \lip(g_n) < \infty.$ 
	
	Then, if $g$ has bounded support, there exists a sequence $(h_n)_n \subset C^k(X)\cap\op{Lip}(X)$ with  bounded supports that converge to $g$ uniformly (on $X$) and moreover $\limsup_n \lip(h_n) < \infty.$
\end{lemma} 

\begin{proof}
	Let $g$ be as in the assumption, and let $R>0$ be so that $g=0$ outside the ball $B(0,R).$ We first construct a suitable bump that decays from one to zero over the annular region $B(0, \lambda R)\setminus B(0,2R),$ for a certain $\lambda>2$ to be fixed later. Since $X$ admits a $C^k\cap\op{Lip}$ bump, by \cite[Proposition II, 5.1]{DevilleGodefroyZizler} there exist a constant $0 < a< 1$ and a function $\psi : X \to [0, \infty)$ so that $\psi \in \lip(X) \cap C^k(X \setminus \lbrace 0 \rbrace)$ and 
	\begin{equation}\label{psicomparabletonorm}
		a \|x\| \leq \psi(x) \leq \|x\|, \quad x\in X.
	\end{equation}
	The statement and proof of \cite[Proposition II, 5.1]{DevilleGodefroyZizler} are written for $k=1,$ but they clearly hold true for $k\in \N \cup \lbrace \infty \rbrace$ as well. Now, pick a function $\theta : \R \to [0,1]$ so that $\theta \in C^\infty(\R) \cap \lip(\R),$ and $\theta(t)=1,$ whenever $t \leq 2R$, and $\theta(t)=0,$ whenever $t \geq 4R.$ Then, define a bump function by
	\begin{equation}\label{suitablebumpfunction}
		\varphi : X \to \R,\qquad \varphi(x)= \theta(\psi(x)), \quad x\in X.
	\end{equation}
	By the properties of $\theta$ and $\psi,$ it is immediate that $\varphi \in C^k(X) \cap \lip(X)$, and that $\varphi$ takes values in $[0,1].$ Also, if $\|x\| \leq 2R,$ one has $\psi(x) \leq 2R$ by \eqref{psicomparabletonorm}, and thus $\varphi(x)=\theta(\psi(x))=1.$ Similary, we deduce that $\varphi(x)=\theta(\psi(x))=0$ for those $x\in X$ such that $\|x\| \geq 4R/a =: \lambda.$

	Now, by assumption we can find a sequence $(g_n)_n \subset C^k(X) \cap \lip (X)$ converging uniformly to $g$ on $X,$ and, after passing to a subsequence, with the property $\sup_n \lip(g_n) < \infty.$ Define $h_n:= \varphi g_n$ for each $n\in \N,$ where $\varphi$ is as in \eqref{suitablebumpfunction}. Since $\varphi$ vanishes outside the ball $B(0,4R/a),$ the function $g_n$ has bounded support. Also,
	\begin{align}\label{eq:laterREF}
		|h_n(x)-g(x)|=|\varphi(x) g_n(x)-g(x)| \leq |\varphi(x)-1| |g_n(x)| + |g_n(x)-g(x)|.
	\end{align}
	The second term in the sum converges to $0$ uniformly on $x.$ For the first term, note that if $\|x\| \leq 2R,$ then $\varphi(x)=1,$ and so $|\varphi(x)-1| |g_n(x)|=0.$ When $\|x\| \geq 2R,$ we have $\lim_n |g_n(x)| = |g(x)|= 0,$ uniformly on those $x,$ by the bounded support of $g.$ Since also $\varphi$ takes values in $[0,1],$ we conclude that the first term tends uniformly to $0$ as $n \to \infty.$ Now we have shown that 
	$$
	\lim_{n\to\infty}\sup_{x\in X} |h_n(x)-g(x)| \to 0.
	$$

	Concerning the regularity of $h_n,$ note that obviously $h_n\in C^k(X)$ because both $\varphi$ and $g_n$ are of class $C^k.$ To verify that $\limsup_n \lip(h_n) < \infty,$ we estimate the derivative of $h_n$ by
	\begin{align*}
		\|Dh_n(x)\|_* \leq \varphi(x) \|Dg_n(x)\|_*  + |g_n(x)| \| D \varphi (x)\|_* \leq \lip(g_n) + \lip(\varphi)( \sup_X |g_n-g|+ \sup_{X} |g| ) .
	\end{align*}
Then $\lip(g_n)\leq \sup_m \lip(g_m),$ $\sup_{X} |g_n-g| \to 0$, and $\sup_{X} |g|<\infty$ because $g$ is Lipschitz with bounded support. This shows $\limsup_n \lip(h_n) < \infty.$ 
\end{proof}

Now, we combine Lemmas \ref{lem:VCsLipApprox}, \ref{lem:Lip-Comeg:Approx}, Theorem \ref{thm:main:LipApprox}, and Lemma \ref{lem:SmoothLipBsApprox}, with an approximation result for separable spaces that admit a $C^{k}\cap\op{Lip}$ bump,
and obtain our most general theorem concerning approximation of $\VC^{0,\omega}$ functions by $C^k\cap\op{Lip}$ functions in separable spaces.

\begin{theorem}[Smooth approximation in separable spaces]\label{thm:main:ApproxSeparable}
	Let $k\in \N \cup\lbrace \infty \rbrace$ and let $X$ be a \textbf{separable} normed space admitting a $C^{k}\cap\op{Lip}$ bump. Then, the following hold:
	\begin{itemize}
		\item[(i)] $\overline{C^k(X)\cap \lip(X) \cap \dot{C}^{0,\omega}(X)}^{\dot{C}^{0,\omega}(X)}=\VC_{\op{small}}^{0,\omega}(X),$  
		\item[(ii)] $ \overline{C^k_{\op{bs}}(X)\cap \mathrm{Lip}(X)}^{\dot{C}^{0,\omega}(X)}=\VC^{0,\omega}(X).$
	\end{itemize}
\end{theorem}
\begin{proof}
	\item[(i)] Let $f\in \VC^{0,\omega}_{\op{small}}(X)$ and $\varepsilon>0.$ By Lemma \ref{lem:VCsLipApprox} we find an approximation $g\in \lip(X)$ so that $\|f-g \|_{\dot{C}^{0,\omega}(X)} \leq \varepsilon.$ Now, since $g: X \to \R$ is a Lipschitz function, and $X$ has a $C^k\cap\op{Lip}$ bump, as a consequence of a result of H\'{a}jek and Johanis \cite[Corollary 15]{HajJoh10} (here we use that $X$ is separable), we find a sequence $(g_n)_n$ of $C^k\cap\op{Lip}$ functions converging uniformly to $g$ such that $\lip(g_n) \leq \lambda \lip(g),$ for all $n,$ and for some absolute constant $\lambda>0$ that may depend on the space $X.$
	Since obviously the function $g$ and the sequence $(g_n)_n$ satisfy all the assumptions of Lemma \ref{lem:Lip-Comeg:Approx}, and hence $\lim_n \|g_n-g\|_{\dot{C}^{0,\omega}(X)}=0,$ completing the proof of the inclusion $``\supset".$
For the reverse inclusion recall that the assumption $t/\omega(t)\to 0,$ as $t\to 0,$ guarantees that $\op{Lip}(X) \cap \dot{C}^{0,\omega}(X)\subset \VC_{\op{small}}^{0,\omega}(X);$ and $\VC_{\op{small}}^{0,\omega}(X)$  is a closed subspace of ${\dot{C}^{0,\omega}(X)}.$

	\item[(ii)] Assume that $f \in \VC^{0,\omega}(X)$ and let $\varepsilon>0.$ According to Theorem \ref{thm:approx:bs} we find $g \in \lip(X)$ with bounded support so that $\|f-g \|_{\dot{C}^{0,\omega}(X)} \leq \varepsilon.$ Now, let  $(g_n)_n$ be the sequence we used in part (i) (but associated with this new $g$), and then Lemma \ref{lem:SmoothLipBsApprox} guarantees a  $C^k\cap\op{Lip}$ approximation with bounded support $(h_n)_n$  so that $\lim_n \sup_X |h_n-g| =0$ and $\limsup_n \lip(h_n) <  \infty.$ Now, we apply Lemma \ref{lem:Lip-Comeg:Approx} to deduce that $\lim_n \|h_n-g\|_{\dot{C}^{0,\omega}(X)} =0$ as well. Thus we pick some $h\in  (h_n)_n\subset C^k(X)\cap \mathrm{Lip}_{\op{bs}}(X)$ such that $\|g-h\|_{\dot{C}^{0,\omega}(X)} \leq \varepsilon,$ which in turn implies
	$$
	\|f-h\|_{\dot{C}^{0,\omega}(X)} \leq \|f-g\|_{\dot{C}^{0,\omega}(X)} + \|g-h\|_{\dot{C}^{0,\omega}(X)} \leq 2 \varepsilon.
	$$
	Thus we have shown the inclusion $``\supset".$
For the converse inclusion $``\subset"$, it is enough to observe again (as we did at the end of the proof of Theorem \ref{thm:main:LipApprox}) that $\mathrm{Lip}_{\op{bs}}(X)\subset \VC^{0,\omega}(X)$ and that $\VC^{0,\omega}(X) $ is a closed subspace of ${\dot{C}^{0,\omega}(X)},$  again, by Remark \ref{remark:VCareclosedsubspaces}.

\end{proof}

\begin{remark}\label{remark:smoothnessofLpspaces} {\em
		Clarifications and remarks concerning Theorem \ref{thm:main:ApproxSeparable} and its proof.
		\begin{enumerate}
			\item[(1)]  If $X$ is a separable Banach space with an equivalent norm of class $C^k$, then Theorem \ref{thm:main:ApproxSeparable} applies for $X$ and $k.$ Indeed, denote such an equivalent norm by $\psi : X \to \R$; and we can assume that $\psi$ satisfies the inequalities on the line \eqref{psicomparabletonorm}. As $\psi \in C^k(X \setminus \lbrace 0 \rbrace),$ and since $\psi$ is subadditive in $X,$ it is clear that $\psi$ is $1$-Lipschitz (with respect to the original norm $\| \cdot \|$) on $X.$ If we now pick $\theta: \R \to [0,1]$ of class $C^\infty \cap \lip$, with $\theta=1$ on $(-\infty, 1/2]$ and $\theta=0$ on $[1,+\infty),$ it is easy to see that the composition $\theta \circ \psi$ defines a $C^k \cap \lip$ bump function on $X.$ Thus, we are in assumptions of Theorem \ref{thm:main:ApproxSeparable} for $X$ and $k.$
			
\item[(2)] Let us now recall the smoothness of the canonical, or equivalent, norms of some classical separable spaces. The canonical norms of the spaces $X= \ell_p$ or $L^p,$ for $1<p<\infty$, have the following order of smoothness $k=k(p):$ 
			\begin{align*}
				k(p) = \begin{cases}
					\infty\qquad & \text{ if } p \text{ is even}, \\
					p-1 ,\qquad & \text{ if } p \text{ is odd},\\
					\left\lfloor p \right\rfloor,\qquad & \text{ if } p \text{ is not an integer}.
				\end{cases}
			\end{align*}
Also, for a compact metric space $K$ which is \textit{scattered},  let $C(K)$ be the Banach space consisting of real-valued continuous functions on $K$, equipped with the supremum norm. Then $C(K)$ admits an equivalent $C^\infty$-norm. Recall that a set $S$ is scattered if every nonempty subset of $S$ contains a (relatively) isolated point. We refer the reader to \cite[Theorems V.1.1 and V.1.8]{DevilleGodefroyZizler} for detailed proofs and statements of these theorems. Then, according to point (1) above, Theorem \ref{thm:main:ApproxSeparable} applies for $X=L^p, \, \ell_p$ with $k_p$ as above, and for $X=C(K)$ with $K$ countable and $k=\infty.$

			\item[(3)] For $X=c_0,$ see Section \ref{subsect:C0GamApprox} below, Theorem \ref{thm:main:ApproxSeparable} also holds for $k=\infty,$ but for this particular $X,$ we will obtain much more in Theorem \ref{thm:main:c0YApprox} below.

			\item[(4)] In the proof of Theorem \ref{thm:main:ApproxSeparable} we employed \cite[Corollary 15]{HajJoh10} for our real-valued functions, but this result holds for target spaces $Y=B_0(V)$ (real-valued functions $f$ in a topological space $V$ with $f(v) \to 0$ as $v \to v_0$, for some fixed $v_0 \in V$); or $Y= C_u(P)$ (uniformly continuous bounded functions on a metric space $P$). Naturally, the particular case $Y=\R$ is covered by these spaces, e.g., when $Y= C_u(P)$ for $P=\lbrace 0 \rbrace \subset \R$. 
		\end{enumerate}
	}
\end{remark}

Now, as a corollary of Theorem \ref{thm:main:ApproxSeparable}, we provide $C^{\infty}\cap\op{Lip}$ approximation of $\VC^{0,\omega}(X)$ functions for separable Hilbert spaces $X$. In particular, notice that this extends Theorem \ref{thm:main:Euc} to infinite dimensions.

\begin{corollary}\label{CorollaryHilbertSeparableApprox}
	Let $X$ be a separable \textbf{Hilbert}  space. Then, the following hold:
	\begin{itemize}
		\item[(i)] $\overline{C^\infty(X)\cap \lip(X) \cap \dot{C}^{0,\omega}(X)}^{\dot{C}^{0,\omega}(X) }=\VC_{\op{small}}^{0,\omega}(X),$  
		\item[(ii)] $ \overline{C^\infty_{\op{bs}}(X)\cap \mathrm{Lip}(X)}^{\dot{C}^{0,\omega}(X)}=\VC^{0,\omega}(X).$
	\end{itemize}
\end{corollary}
\begin{proof}
	Both statements (i) and (ii) follow immediately from Theorem \ref{thm:main:ApproxSeparable}, since a Hilbert space always admits a $C^\infty\cap\op{Lip}$ bump. For the sake of completeness, we exhibit an elementary construction of such a bump.
	
	The function $\psi: X \to \R$ given by $x \mapsto \psi(x)=\|x\|^2$ is of class $C^\infty(X)$ with the Fr\'{e}chet derivative $D\psi (x): X \to \R,$ at every $x\in X$, given by $D\psi(x)(v) = \langle 2x, v \rangle $ for all $v\in X.$ Take a function $\theta: \R \to [0,1]$ of class $C^\infty(\R) \cap \lip(\R)$ with $\theta(t)=1$ for all $t \leq 1/2$ and $\theta(t)=0$ for all $t \geq 1.$ Then, the map $\varphi = \theta \circ \psi : X \to [0,1]$ is of class $C^\infty(X)$ with $\varphi(x)=1$ for all $\|x \| \leq 1/2$ and $\varphi(x) = 0$ for all $\|x \| \geq 1.$ Also, the Fr\'{e}chet derivative of $\varphi$ satisfies
	$$
	\|D\varphi(x)\|_*= \|\theta'(\psi(x)) \Chi_{B(0,1)}(x) 2x  \|_* \leq 2 \lip(\theta),
	$$
	showing that $\varphi$ is Lipschitz on $X.$ Therefore, $\varphi$ is a $ C^\infty\cap\op{Lip} $ bump on $X.$ 
\end{proof}

\subsection{Approximation in super-reflexive spaces}\label{subsect:SupRefApprox}

Let us now approximate real-valued functions of the classes $\VC^{0,\omega}_{\op{small}}(X)$ and $\VC^{0,\omega}(X)$ over a super-reflexive Banach space $X.$ We remind that $X$ is super-reflexive if every Banach space $Y$ that is \textit{finitely representable} into $X$ is reflexive. According to Pisier's renorming theorem \cite{Pisier}, a Banach space $X$ is super-reflexive if and only if $X$ admits a renorming (which we keep denoting by $\| \cdot\|$), an exponent $\alpha \in (0,1]$, and a constant $C>0$ for which
\begin{equation}\label{C1alpharenorming}
	\|x+h\|^{1+\alpha} + \|x-h\|^{1+\alpha} - 2 \|x\|^{1+\alpha} \leq C \|h\|^{1+\alpha}, \quad \text{for all} \quad x,h\in X.
\end{equation}
See also \cite[pp. 412--413]{BenyaminiLindenstrauss} for a proof of this equivalence. This property is often rephrased by saying that $X$ admits a renorming with \textit{modulus of smoothness of power type} $1+\alpha.$ Since the function $ x\mapsto\psi(x) = \|x\|^{1+\alpha}$ is convex and continuous on $X,$ then \eqref{C1alpharenorming} implies that $\psi$ is of class $C^{1,\alpha}(X);$ and with $\|D \psi\|_{\dot{C}^{0,\alpha}(X,X^*)} \lesssim_{\alpha,C} 1$; see \cite[Lemma V.3.5]{DevilleGodefroyZizler}. 

We remind that a function $f: X \to \R$ belongs to the class $C^{1,\alpha}(X)$ if $f$ is Fr\'{e}chet differentiable at every point $x\in X$ and the Fr\'{e}chet derivative $Df: X \to X^*$ is $\alpha$-H\"older continuous on $X,$ namely, that
$$
\|Df\|_{\dot{C}^{0,\alpha}(X,X^*)}:= \sup \left\lbrace \frac{\|Df(x)-Df(y)\|_*}{\|x-y\|^\alpha} \, : \, x,y\in X, \, x \neq y \right\rbrace < \infty.
$$

In Hilbert spaces the Lasry--Lions regularization theorem (see \cite{LasryLions} or \cite[p. 408]{HajJoh10}) provides uniform approximation of Lipschitz functions by $C^{1,1}\cap\op{Lip}$ functions. Cepedello-Boiso \cite[Theorem 1]{Cepedello} used an ingenious variant of the Lasry-Lions technique to obtain uniform approximation of Lipschitz functions in super-reflexive spaces by $C^{1}$ functions whose derivatives are $\alpha$-H\"older on bounded subsets, but not globally on $X.$ These approximations are, of course, Lipschitz on bounded sets, but not globally Lipschitz. However, for our purposes, we need to approximate Lipschitz functions uniformly by \textit{globally} $C^{1,\alpha}\cap\op{Lip}$ functions, and we need a good control on the Lipschitz constants of the approximations. In Theorem \ref{thm:main:approxLipssuper-reflexive}(2) below, we have obtained uniform approximations that are \textit{globally} $C^{1,\alpha}$ and \textit{globally} Lipschitz on $X$, preserving the Lipschitz constants up to a multiplicative factor depending only on $\alpha$ and $X.$ In the proof, we use a recent result on $C^{1,\omega}$ extensions for jets \cite{AzaMud21} to construct $C^{1,\alpha}$ \textit{bump functions} with all the properties stated in part (1) of Theorem \ref{thm:main:approxLipssuper-reflexive}. This is combined with a method of H\'{a}jek and Johanis \cite[Proposition 1]{HajJoh10} to \textit{glue up} a suitable sequence of these bumps. Both parts (1) and (2) in Theorem \ref{thm:main:approxLipssuper-reflexive} are interesting to the theory of smooth approximation, and also will be essential for the approximation of $\VC^{0,\omega}$ functions.

Let us mention that, after making the first version of this article public, it has come to our attention that part (2) of Theorem \ref{thm:main:approxLipssuper-reflexive} was very recently discovered by Johanis \cite{Jo24} too, by means of a different proof. A benefit of Johanis' proof is that it provides the sharp constant $\kappa = 1.$ Since we only need an absolute control on $\kappa$ for our purposes, and since we consider part (1) of independent interest as well, we have chosen to include our original proof of Theorem \ref{thm:main:approxLipssuper-reflexive}(2).

\begin{theorem}\label{thm:main:approxLipssuper-reflexive}  
Let $X$ be a super-reflexive space, let $\alpha \in (0,1]$ and $C>0$ be so that a renorming of $\| \cdot \|$ of $X$ satisfies \eqref{C1alpharenorming}. Then, there exists a constant $\kappa  \geq 1$ depending only on $\alpha$ and $C$ for which the following hold. 
	\begin{enumerate}
		\item  For every set $S \subset X$ there exists $h_S: X \to [0,1]$ of class $C^{1,\alpha}(X)\cap\op{Lip}(X)$ so that: \\
		$h_S(x)=0,$ for every $x\in S$; \\
		$h_S(x)=1,$ whenever $d(x,S) \geq 1;$ \\
		$D h_S(x)=0,$ for every $x\in  S \cup \lbrace y\in X \, : \, d(y,S) \geq 1 \rbrace;$ 
		and
		$$
		\lip(h_S) + \|Dh_S\|_{\dot{C}^{0,\alpha}(X,X^*)} \leq \kappa .
		$$
		\item  Given an $L$-Lipschitz function $g: X \to \R$ and $\varepsilon>0,$ there exists $h\in C^{1,\alpha}(X)$ so that $h$ is $  \kappa  L$-Lipschitz, and $\sup_X |g-h| \leq \varepsilon.$
	\end{enumerate}
	
\end{theorem}

\begin{proof}[Proof of Theorem \ref{thm:main:approxLipssuper-reflexive}]
	\item[(1)] In the cases $S= \emptyset$ or $\lbrace \dist(\cdot, S) \geq 1 \rbrace = \emptyset,$ we simply take $h_S \equiv 1$ in the first case, and $h_S \equiv 0$ in the latter. Assume from now on that both sets are nonempty. On the set 
	$$
	E:=S \cup \lbrace x\in X \, : \, d(x,S) \geq 1 \rbrace,
	$$ 
	we define a $1$-jet $(f,G): E \to \R \times X^*$ by setting
	$$
	f(x)=0, \quad \text{if} \quad x\in S; \qquad f(x)=1, \quad \text{if} \quad d(x,S) \geq 1; \qquad G(x)=0, \quad \text{for all} \quad x\in E.
	$$
	By separating into cases, it is immediate to verify the existence of $M=M(\alpha)>0$ for which
	$$
	f(y)+G(y)(x-y) -f(z)-G(z)(x-z)  \leq   \frac{M}{1+\alpha} \left(  \|x-y\|^{1+\alpha}+ \|x-z\|^{1+\alpha} \right),
	$$
	for all $y,z \in E,$ $x\in X;$ indeed, $G(y) = G(z) = 0$ is the zero functional, and for $f(y)-f(z),$ by symmetry, it is enough to verify the case $y\in S$ and $\dist(z,S)\geq 1$ and $x\in X$ arbitrary, but this case is clear by inspection, since 
		\[
		f(y)-f(z) = 1 \leq \| y-z\|\leq  \| y-z\|^{1+\alpha}\leq \frac{M}{1+\alpha} \Big(  \|x-y\|^{1+\alpha}+ \|x-z\|^{1+\alpha} \Big).
		\] Thus, applying \cite[Theorem 4.1]{AzaMud21}, there exists a constant $\kappa_0=\kappa_0(\alpha,C)$, and a function $F: X \to \R$ of class $C^{1,\alpha},$ bounded and Lipschitz, with 
	\begin{equation}\label{eq:ExtensionProperties}
		\|F\|_\infty + \lip(F) + \|DF\|_{\dot{C}^{0,\alpha}(X,X^*)} \leq \kappa_0 M=: \kappa_1,
	\end{equation}	
and so that $F=f$ and $DF=G=0$ on $E.$ Although \cite[Theorem 4.1]{AzaMud21} is formulated for Hilbert spaces, it generalizes immediately to super-reflexive Banach spaces satisfying property \eqref{C1alpharenorming}, as pointed out in \cite[Remark 4.7]{AzaMud21}; see also \cite[Theorem 1.9]{AzaMud21}. Let us examine the extension formula $F$, since we need some of its properties.  At the same time, for the reader's convenience, we sketch some of the key steps in the proof of \cite[Theorem 4.1]{AzaMud21}. With $M$ as above, 
	$$
	M^* :=\max\Big\{ 3 \big( \|f\|_{\infty}+  \|G\|_{\infty} \big), \, \frac{M}{1+\alpha}\Big\},
	$$
	and the functions $m,g:X \to \R $ by
	\begin{align*}
		m(x) &:=\max\Big\{ -2\Big(\|f\|_\infty+ \|G\|_{\infty}\Big), \, \sup_{z\in E}\Big\{f(z)+ G(z)(x-z) -   M^*\|x-z\|^{1+\alpha} \Big\} \Big\}, \\
		g(x) &:=\min\Big\{ 2\Big(\|f\|_{\infty}+\|G\|_{\infty}\Big), \, \inf_{y\in E}\Big\{f(y)+  G(y) (x-y) + M^*\|x-y\|^{1+\alpha} \Big\}\Big\}.
	\end{align*} 
	For a suitable number $A$ (depending only on $\alpha$ and $C$), $F$ is defined to be the $A M^*t^{1+\alpha}$-\textit{strongly paraconvex envelope of} $g$, that is,
	$$
	F(x):=\sup\lbrace h(x) \, : \, h \text{ is } A M^*t^{1+\alpha}\text{-strongly paraconvex,   and   } h \leq g \text{ on } X \rbrace,\quad x\in X.
	$$
As defined in \cite{AzaMud21}, a function $h: \psi \to \R$ is $ L t^{1+\alpha}$-strongly paraconvex, for $L>0,$ if
$$
\psi(\lambda u + (1-\lambda)v)- \lambda \psi(u)-(1-\lambda)\psi(v) \leq   \lambda (1-\lambda) L \|u-v\|^{1+\alpha}, \quad u,v\in X,\, \lambda \in [0,1].
$$
By the comment just after formula \eqref{C1alpharenorming}, the function $u \mapsto \psi(u):= -\|u\|^{1+\alpha}$ is $C^{1,\alpha}$, and their Fr\'{e}chet derivative $D \psi$ satisfies $\|D \psi \|_{\dot{C}^{0,\alpha}(X,X^*)} \leq B(\alpha,C).$ Thus, using the Fundamental Theorem of Calculus and some elementary computations, $\psi$ is $A t^{1+\alpha}$-strongly paraconvex for some $A=A(\alpha,C);$ see the argument near the end of \cite[the proof of Lemma 3.6]{AzaMud21}. It turns out that then both $m$ and $(-g)$ are $A M^*t^\alpha$-strongly paraconvex. Also, $m \leq g$ on $X,$ and these properties permit to prove the regularity $F\in C^{1,\alpha}(X)$, along with \eqref{eq:ExtensionProperties} and that $(F,DF)=(f,G)$ on $E.$ See \cite[Theorems 1.9 and 4.1]{AzaMud21} for further explanations and details of these concepts and their proofs. Now, by the definition of $F,$ the pointwise relation
	$$
	m(x) \leq F(x) \leq g(x), \quad x\in X,
	$$
holds true. Since $G\equiv 0$ and $\|f\|_\infty=1,$ the definition of $m$ and $g$ and this estimate imply that $-2 \leq F \leq 2$ on $X.$  It only remains to slightly modify $F$ so that the final function $h_S$ takes values in $[0,1]. $

To do so, we pick a bump $\theta: \R \to [0,1]$ of class $C^\infty(\R) \cap C^{1,\alpha}(\R) \cap \lip(\R)$ with $\theta(t)=0$ for $t \leq 0$ and $t \geq 2,$ and $\theta(1)=1.$ Define $h_S:=\theta \circ F :X \to [0,1]. $ It is immediate that $h_S(x)=0$ for every $x\in S$, $h_S(x)=1$ if $d(x,S) \geq 1,$ and that $Dh_S=0$ on $E.$ Since $F$ is $\kappa_1$-Lipschitz, $h_S$ is $\lip(\theta) \kappa_1$-Lipschitz. Finally, using the facts that $\lip(F) + \|DF\|_{\dot{C}^{0,\alpha}(X,X^*)} \leq \kappa_1$ and that $\lip(\theta) + \|\theta'\|_{\dot{C}^{0,\alpha}(\R)} \leq c(\alpha)$ for some $c(\alpha)>0,$ it is an easy exercise to verify that
$$
\|Dh_S\|_{\dot{C}^{0,\alpha}(X,X^*)} \leq \|\theta'\|_{\dot{C}^{0,\alpha}(\R)} \lip(F)^\alpha + \lip(\theta) \| DF\|_{\dot{C}^{0,\alpha}(X,X^*)} \leq \kappa(\alpha,\kappa_1).
$$ 
Therefore, we can find $\kappa=\kappa(\alpha,C)$ for which $h_S$ satisfies all the properties stated in (1).

	\medskip
	\item[(2)] Let $g:X \to \R$ be $L$-Lipschitz, and $\varepsilon>0.$ We will now apply part (1) for a suitable sequence of sets to obtain the desired approximation. In order to do so, we use the same construction as in \cite[Proposition 1]{HajJoh10}, replacing the $C^1$-\textit{separating functions} from there with those we just obtained in part (1).

	Define $\tilde{g}(x) = \varepsilon^{-1} g( \varepsilon x /L),$ for all $x\in X,$ and note that $\tilde{g}$ is $1$-Lipschitz. Now, we define the set $S_n= \lbrace x\in X\, : \,\tilde{g}(x) \geq n \rbrace,$ for each integer $n \in \Z.$ Note that $S_{n+1} \subset S_n$, and also, because $\tilde{g}$ is $1$-Lipschitz, one has that $d(X\setminus S_n , S_{n+1} ) \geq 1$ for every $n\in \Z.$ Now, for each subset $S_n,$ let $h_{S_{n}}: X \to [0,1]$ be the function from (1). Define now $h_n = 1- h_{S_{n+1}}$ for every $n$, 
	and also
	$$
	h(x)= \sum_{n=0}^\infty h_n(x) + \sum_{n=-\infty}^{-1} (h_n(x)-1), \quad x\in X.
	$$
Each $h_{n}$ is $C^1$ and $\kappa $-Lipschitz, $h_{n}=1$ on $S_{n+1},$ and $h_{n}(x)=0$ if $d(x, S_{n+1}) \geq 1.$ As proven in \cite[Proposition 1]{HajJoh10}, $h: X \to \R$ is a well-defined $\kappa $-Lipschitz and $C^1$ function, with $\sup_X |h-\tilde{g}| \leq 1.$ Moreover, the sums defining $h$ are locally finite, meaning that for every $x\in X,$ there exists $N_x\in \N,$ and a ball $B(x,r_x)$ so that
\begin{equation}\label{eq:difftermbyterm}
h(z)= \sum_{n=0}^{\max \{|m|,N_x \}} h_n(z) + \sum_{n=-\max\{ |m|, N_x \}}^{-1} (h_n(z)-1), \quad z\in B(x,r_x),\, m\in \Z.
\end{equation}
The relation \eqref{eq:difftermbyterm} will allow us to differentiate the sums defining $h$ term by term locally around every $x\in X.$

	Let us now prove that $h\in C^{1,\alpha}(X)$ with $\|Dh\|_{\dot{C}^{0,\alpha}(X,X^*)} \leq 2\kappa.$ Recalling that by definition $Dh_{S_{n+1}}(x)=0$ whenever $x\in S_{n+1} \cup \lbrace x\in X \, : \, d(x,S_{n+1}) \geq 1 \rbrace$, it is immediate from the definition  $h_n := 1- h_{S_{n+1}}$ that
	$$
	Dh_n=0 \quad \text{on} \quad S_{n+1} \cup \lbrace x \in X \, : \, d(x,S_{n+1}) \geq 1 \rbrace, \quad n\in \Z.
	$$
	Note that this implies $Dh_n=0$ on $X\setminus S_n,$ as $d(X\setminus S_n , S_{n+1}) \geq 1.$ These observations together with the continuity of $Dh_n$ give
	\begin{equation}\label{derivativeNullboundary}
		Dh_n=0 \quad \text{on} \quad S_{n+1} \cup \overline{X\setminus S_n}, \quad n \in \Z.
	\end{equation} 
Also, we claim that
	\begin{equation}\label{derivativeontheannuli}
		Dh(x)=Dh_m(x), \quad \text{whenever} \quad x\in S_m \setminus S_{m+1}, \quad m \in \Z.
	\end{equation}
	Indeed, let $x\in S_m \setminus S_{m+1}$. For those $n \geq m+1,$ we have $x\notin S_n $ and thus $Dh_n(x)=0,$ e.g., by virtue of \eqref{derivativeNullboundary}. And for those $n \leq m-1, $ we have $x\in S_{n+1}$ and so $Dh_n(x)=0.$ In other words, $Dh_n(x)=0$ for $n \neq m.$ It then follows by \eqref{eq:difftermbyterm} and evaluating at $x\in S_m \setminus S_{m+1}$ that
	$$
	Dh(x)=\sum_{n=0}^{\max \{|m|,N_x \}} D h_n(x) + \sum_{n=-\max \{|m|,N_x \}}^{-1} Dh_n(x) = Dh_m(x).
	$$
This proves the claim \eqref{derivativeontheannuli}.

	Now given $x,y \in X,$ let $m \in \Z$ and $l \in \N \cup \lbrace 0 \rbrace$ be so that $x\in S_m \setminus S_{m+1} $ and $y\in S_{m+l} \setminus S_{m+l+1},$ and we show the $\alpha$-H\"older estimate for $\|Dh(x)-Dh(y)\|_*.$

	In the case $l=0,$ we have $x,y\in S_m \setminus S_{m+1},$ and so it suffices to apply \eqref{derivativeontheannuli} and the fact that $\|Dh_m\|_{\dot{C}^{0,\alpha}(X,X^*)} \leq \kappa.$

	Assume now $l \geq 1$. By the connectedness of $[x,y]$ and $x\not\in S_{m+l}$ and $y\in S_{m+l}$, the segment $[x,y]$ must intersect the boundary of $S_{m+l}.$ Let $z\in [x,y] \cap  S_{m+l} \cap \overline{X \setminus S_{m+l}}$. In particular, $z\in  \overline{X \setminus S_{m+l}},$ and by \eqref{derivativeNullboundary}, this implies $Dh_{m+l}(z)=0.$ Also, observe that $y\in S_{m+l} \subset S_{m+1},$ and so $Dh_m(y)=0,$ again by \eqref{derivativeNullboundary}. Using first \eqref{derivativeontheannuli}, then $Dh_m(y)=Dh_{m+l}(z)=0$, and finally that $\|Dh_n\|_{\dot{C}^{0,\alpha}(X,X^*)} \leq \kappa $ for all $n \in \Z,$ we conclude that
	\begin{align*}
		\|Dh(x)-Dh(y)\|_* & = \|Dh_m(x)-Dh_{m+l}(y)\|_* \\
		& \leq \|Dh_m(x)-Dh_m(y)\|_* + \|Dh_{m+l}(z)-Dh_{m+l}(y)\|_* \\
		& \leq \kappa  \left( \|x-y\|^{\alpha} + \|z-y\|^{\alpha} \right) \leq 2 \kappa  \|x-y\|^{\alpha}.
	\end{align*}
By symmetry, we can swap the roles of $x$ and $y,$ and thus we have shown that $h\in C^{1,\alpha}(X).$ 

Finally, rescaling the function $h$ by $\tilde{h}(x)= \varepsilon h( xL/\varepsilon)$ and noting that $g(x) = \varepsilon \tilde{g}( xL/\varepsilon),$ we see that $\tilde{h} \in C^{1,\alpha}(X)$ with $\sup_X |\tilde{h}-g| \leq \varepsilon$ and $\lip(\tilde{h}) \leq L \lip(h) \leq \kappa L.$
\end{proof}

We will use the following elementary fact.
\begin{remark}\label{remark:easyproperty}
	{\em If $U,V$ are normed spaces, and $\psi: U \to V$ is Lipschitz and bounded, then $\psi$ is $\alpha$-H\"older for every $\alpha \in (0,1].$}
\end{remark}

A consequence of Theorem \ref{thm:main:approxLipssuper-reflexive}, in combination with our results from previous sections, is the following $C^{1,\alpha}\cap\op{Lip}$ approximation of $\VC^{0,\omega}$ functions in super-reflexive spaces. Note that, unlike in Theorem \ref{thm:main:ApproxSeparable}, $X$ may be nonseparable. 

\begin{theorem}[Approximation in super-reflexive spaces]\label{ThmApproxsuper-reflexive}
	Let $X$ be a super-reflexive space, and let $\alpha \in (0,1]$ and $C>0$ be so that a renorming of $\| \cdot \|$ of $X$ satisfies \eqref{C1alpharenorming}. Then the following hold: 
	\begin{itemize}
		\item[(i)]  $  \overline{C^{1,\alpha}(X) \cap \mathrm{Lip}(X) \cap \dot{C}^{0,\omega}(X)}^{\dot{C}^{0,\omega}(X)} = \VC^{0,\omega}_{\op{small}}(X),$ 
		\item[(ii)] $  \overline{C^{1,\alpha}_{\op{bs}}(X)}^{\dot{C}^{0,\omega}(X)} =\VC^{0,\omega}(X). $
	\end{itemize}
	
\end{theorem}
Notice that the modulus $\omega$ is independent of $\alpha.$ Also,  the approximations in part (ii) are Lipschitz. Indeed, as the derivative is $\alpha$-H\"older and boundedly supported, it is in particular bounded and this implies that the function is Lipschitz.

\begin{proof}
	\item[(i)] Let $f\in \VC^{0,\omega}_{\op{small}}(X)$ and $\varepsilon>0.$ Applying Lemma \ref{lem:VCsLipApprox}, we can find a function $g\in \lip(X)$ so that $\|f-g \|_{\dot{C}^{0,\omega}(X)} \leq \varepsilon.$ Now, by Theorem \ref{thm:main:approxLipssuper-reflexive}(2), there is a sequence $(g_n)$ of $\op{Lip}(X)\cap C^{1,\alpha}(X)$ functions with $\sup_n \lip(g_n) \leq \kappa \lip(g) ,$ and $\lim_n   \sup_X |g-g_n| =0.$ Therefore Lemma \ref{lem:Lip-Comeg:Approx} yields that $\lim_n \|g_n-g\|_{\dot{C}^{0,\omega}(X) } = 0,$ and taking some $g_n$ from the sequence with $ \|g_n-g\|_{\dot{C}^{0,\omega}(X) } \leq \varepsilon,$ we conclude the inclusion $``\supset "$ of $(i).$ The reverse inclusion follows by noting that $\op{Lip}(X) \cap \dot{C}^{0,\omega}(X) \subset \VC^{0,\omega}_{\op{small}}(X),$ by the condition \eqref{eq:mod:coer1}, and that $\VC^{0,\omega}_{\op{small}}(X)$ is a closed subspace of $\dot{C}^{0,\omega}(X).$

	\item[(ii)] Let $f\in \VC^{0,\omega}(X)$ and $\varepsilon>0.$ By Theorem \ref{thm:main:LipApprox} we find $g \in \lip(X)$ with bounded support so that $\|f-g \|_{\dot{C}^{0,\omega}(X)} \leq \varepsilon.$ With $R>0$ so that $g(x)=0$ for $\| x \|\geq R,$ we use Theorem \ref{thm:main:approxLipssuper-reflexive}(1) to obtain a $\op{Lip}(X)\cap C^{1,\alpha}(X)$ function $\varphi: X \to [0,1]$ with $\varphi=1$ on $B(0,2R)$ and $\varphi=0$ on $X \setminus B(0,2R+1).$ If $(g_n)_n $ is the approximating function for $g$ from part (i), we define $h_n= \varphi g_n.$ Obviously each $h_n \in C^1(X)$ and also has bounded support, and moreover the properties $\|h_n-g\|_\infty \to 0$ and $\limsup_n \lip(h_n) < \infty$ hold and are checked exactly as in the proof of Lemma \ref{lem:SmoothLipBsApprox}, see the line \eqref{eq:laterREF}. So, it only remains to verify that $h_n$ has $\alpha$-H\"older derivative, at least, for $n$ large enough. Write,
	$$
	Dh_n(x)= g_n(x)D \varphi(x) + \varphi(x) Dg_n(x) 
	$$
	and let us prove that both terms define an $\alpha$-H\"older function. 
	
	For the first function $x\mapsto g_n(x) D \varphi(x) $, let us observe that, for $n$ large enough, one has $\sup_X  |g_n | \leq 1+ \sup_X |g|,$ as $g_n $ converges uniformly on $X$ to $g.$ Hence, by Remark \ref{remark:easyproperty}, $g_n$ is $\alpha$-H\"older continuous on $X.$ Then, we write
		\begin{align*}
		\|g_n(x) D \varphi(x) & - g_n(z) D \varphi(z)\|_*  \\
		& \leq |g_n(x)| \| D \varphi(x)- D\varphi(z) \|_* +  | g_n(x)-g_n(z)| \|D \varphi(z)\|_* \\
		& \leq  \sup_X  |g_n | \| D \varphi \|_{\dot{C}^{0,\alpha}(X,X^*)} \|x-z\|^\alpha + \| g_n \|_{\dot{C}^{0,\alpha}(X,\R)} \|x-z\|^\alpha  \sup_X \|D \varphi\|_*.
\end{align*}
All the factors that multiply the term $\|x-z\|^\alpha$ are finite by the properties of $\varphi$ and $g_n.$

	As concerns the second function $x\mapsto \varphi(x) Dg_n(x) $, its $\alpha$-H\"older continuity is verified in a very similar way, this time using Remark \ref{remark:easyproperty} for $\varphi$ and that $Dg_n$ is bounded (as $g_n$ is Lipschitz) and $\alpha$-H\"older continuous.

	We conclude that $(h_n)_n \to g$ uniformly, with $(h_n)_n \in C^{1,\alpha}_{\op{bs}}(X)$ and $\limsup_n \lip(h_n) < \infty.$ By Lemma \ref{lem:Lip-Comeg:Approx}, we have $\|h_n-g\|_{\dot{C}^{0,\omega}(X)} \to 0$ as well. Therefore, we can find $h\in C^{1,\alpha}_{\op{bs}}(X)  $ with $\|h-g\|_{\dot{C}^{0,\omega}(X)} \leq \varepsilon,$ and thus $\|h-f\|_{\dot{C}^{0,\omega}(X)} \leq 2\varepsilon.$

	We have shown the inclusion $\VC^{0,\omega}(X) \subset \overline{C^{1,\alpha}_{\op{bs}}(X)}^{\dot{C}^{0,\omega}(X)}.$ For the reverse inclusion, observe that if $f\in C^{1,\alpha}_{\op{bs}}(X),$ then $Df: X \to X^*$ is $\alpha$-H\"older and with bounded support, and so $Df$ is bounded in $X.$ Because Lipschitz functions with bounded support are contained in $\VC^{0,\omega}(X),$ which is closed under limits with respect to the $\dot{C}^{0,\omega}(X)$ seminorm, the reverse inclusion holds.  
\end{proof}

\begin{remark}\label{rem:smoothnessL^psuper}
	{\em We exhibit some examples of classical spaces that are covered by Theorem \ref{ThmApproxsuper-reflexive}.
		
		\begin{enumerate}
			\item If $X=L^p,$ for $1<p< \infty,$ then \eqref{C1alpharenorming} holds for $\alpha_p=p-1$ when $p \leq 2$ and for $\alpha_p=1$ when $p \geq 2$; for a proof see \cite[Corollary V.1.2]{DevilleGodefroyZizler}. By Theorem \ref{ThmApproxsuper-reflexive}, this gives rise to approximations of $\VC^{0,\omega}_{\op{small}}(X)$ and $\VC^{0,\omega}(X)$ by Lipschitz functions of class $C^{1,\alpha_p}(X)$. 
				\item If $X$ is a Hilbert space (separable or not), then \eqref{C1alpharenorming} holds as an identity for $\alpha=1$ and  $C=2$  (the parallelogram law), and so Theorem \ref{ThmApproxsuper-reflexive} gives $C^{1,1}$ approximations:
				\begin{itemize}
					\item[(i)]  	$\VC^{0,\omega}_{\op{small}}(X)=\overline{C^{1,1}(X)\cap \lip(X) \cap \dot{C}^{0,\omega}(X) }^{\dot{C}^{0,\omega}(X)},$
					\item[(ii)]  	 $\VC^{0,\omega}(X)= \overline{C^{1,1}_{\op{bs}}(X)}^{\dot{C}^{0,\omega}(X)}.$
			\end{itemize}
		\end{enumerate}
	}
\end{remark}

\medskip

\subsection{Approximation of Banach-valued mappings from $c_0$}\label{subsect:C0GamApprox} 
For an arbitrary set of indices $\mathcal{A},$ the space $c_0(\mathcal{A})$ consists of those elements $x=(x_\alpha)_{\alpha \in \mathcal{A}}\in\R^{\mathcal{A}}$ such that for every $\varepsilon>0,$ there exists a finite subset $S = S_x\subset \mathcal{A}$ so that $|x_\alpha| \leq \varepsilon,$ whenever $\alpha \in \mathcal{A} \setminus S.$  
We equip $c_0(\mathcal{A})$ with the norm $\|(x_\alpha)_{\alpha \in \mathcal{A}}\|_\infty:= \sup_{\alpha \in \mathcal{A}} | x_\alpha|$ and this results in a Banach space. 

Given $x=(x_\alpha)_{\alpha \in \mathcal{A}}$ and $\alpha\in\mathcal{A},$ we denote the projection $(x)_{\beta}  = ((x_\alpha)_{\alpha \in \mathcal{A}})_{\beta} := x_{\beta}\in\R;$  and for every finite subset $S \subset \mathcal{A}, $ we denote
$$
P_S: c_0(\mathcal{A}) \to c_{00}(\mathcal{A}),\qquad P_S(x)= \sum_{\alpha \in S} x_\alpha e_\alpha,
$$
where $e_\alpha$ is the element of $c_0(\mathcal{A})$ that satisfies $(e_\alpha)_{\beta}=\delta_{\alpha,\beta},$ for every $\beta \in \mathcal{A}.$

\begin{theorem}[Approximation in $c_0(\mathcal{A})$]\label{thm:main:c0YApprox}
	For an arbitrary set of indices $\mathcal{A},$ let $X=c_0(\mathcal{A}),$ and let $Y$ be a Banach space. Then, the following hold:
	\begin{itemize}
		\item[(i)] $\overline{C^{\infty}(X,Y)\cap \VC^{0,\omega}_{\op{small}}(X,Y)}^{\dot{C}^{0,\omega}(X,Y)}= \VC^{0,\omega}_{\op{small}}(X,Y),$
		\item[(ii)] $\overline{C^{\infty}_{\op{bs}}(X,Y)\cap \VC^{0,\omega}(X,Y)}^{\dot{C}^{0,\omega}(X,Y)}=\VC^{0,\omega}(X,Y).$
	\end{itemize}
	Moreover, in the particular case $Y=\R,$ the approximation can be taken to be Lipschitz:
	\begin{itemize}
	\item[(i)] $\overline{C^{\infty}(X)\cap \lip(X) \cap \dot{C}^{0,\omega}(X)}^{\dot{C}^{0,\omega}(X)} = \VC^{0,\omega}_{\op{small}}(X),$
	\item[(ii)] $  \overline{C^{\infty}_{\op{bs}}(X)\cap \lip(X)}^{\dot{C}^{0,\omega}(X)} = \VC^{0,\omega}(X).$
\end{itemize}
\end{theorem}

One step in our proof of Theorem \ref{thm:main:c0YApprox} will rely on the construction by H\'{a}jek and Johanis \cite[Theorem 1]{HajJoh09} (stated for Lipschitz functions), to obtain the $C^\infty$ approximation on $c_0(\mathcal{A}).$ However, the verification $\dot{C}^{0,\omega}$ convergence after our $\VC^{0,\omega}_{\op{small}}$ condition will require more work. 

Following \cite{HajJoh09} we first construct an approximation that locally depends only on a finite number of coordinates. 

\begin{lemma}\label{lem:preliminaryforc0}
	Let $f\in \VC^{0,\omega}_{\op{small}}(X,Y)$ and $\varepsilon>0.$ Then, there exists $g\in \VC^{0,\omega}_{\op{small}}(X,Y)$ and $r>0$ so that $\|g-f\|_{\dot{C}^{0,\omega}(X,Y)} \leq \varepsilon,$ and for every $x\in X,$ there exists a finite subset $S=S_x$ of $\mathcal{A}$ for which $g(z)=g(P_S(z)),$ whenever $z\in B(x,r).$ 
	
	Moreover, if $f$ has bounded support (resp. Lipschitz), then the above $g$ has bounded supported (resp. Lipschitz) too.
\end{lemma}
\begin{proof}
	Let $f\in \VC^{0,\omega}_{\op{small}}(X,Y)$ and $\varepsilon>0,$ let $\delta>0$ be so that
	\begin{equation}\label{eq:proofprelimc0}
		\sup_{\|u-v\|\leq \delta} \frac{\|f(u)-f(v)\|_Y}{\omega(\|u-v\|)} \leq \frac{\varepsilon}{2}.
	\end{equation}
	By $ \omega(0)=0,$ we find $r>0$ so that $2\omega(r)\| f\|_{\dot C^{0,\omega}(X,Y)} \leq \varepsilon \omega(\delta).$ Let us define 
	$$
	\varphi :  \R \to \R,\qquad \varphi (t)   =  \left\lbrace
	\begin{array}{ccl}
		t+r & \mbox{if } & t \leq - r, \\
		0  & \mbox{if }&  -r \leq t\leq r, \\
		t-r & \mbox{if }& t \geq  r,
	\end{array}
	\right.
	$$
	and 
	$$
	\phi : X \to X,\qquad \phi(x)= \phi((x_\alpha)_{\alpha \in \mathcal{A}}) =( \varphi(x_\alpha) )_{\alpha \in \mathcal{A}}.
	$$
	As clearly $\varphi$ is $1$-Lipschitz, it is immediate that $\phi$ is $1$-Lipschitz. Also, it is immediate that $ \|f-f\circ\phi\|_Y \leq \omega(r)\| f\|_{\dot C^{0,\omega}(X,Y)}.$  Defining $g:= f \circ \phi,$ we have $g \in \VC^{0,\omega}_{\op{small}}(X,Y),$ by virtue of Lemma \ref{lem:LipSmall}. If $x,z\in X$ are such that $\|x-z\| \leq \delta,$ then $\|\phi(x)-\phi(z)\| \leq \delta$ and so 
	$$
	\|g(x)-g(z)\|_Y= \| f(\phi(x))-f(\phi(z))\|_Y \leq \frac{\varepsilon}{2} \omega(\|\phi(x)-\phi(z)\|) \leq \frac{\varepsilon}{2} \omega(\|x-z\|).
	$$
	This estimate, in combination with \eqref{eq:proofprelimc0}, yields,
	$$
	\sup_{\|x-z\|\leq \delta} \frac{\|(f-g)(x)-(f-g)(z)\|_Y}{\omega(\|x-z\|)} \leq \frac{\varepsilon}{2} + \frac{\varepsilon}{2} = \varepsilon.
	$$
	On the other hand, the choice of $r$ permits to estimate by
	$$
	\sup_{\|x-z\|\geq \delta} \frac{\|(f-g)(x)-(f-g)(z)\|_Y}{\omega(\|x-z\|)} \leq \frac{2\|f-g\|_\infty}{\omega(\delta)}  \leq  \frac{2\omega(r)\| f\|_{\dot C^{0,\omega}(X,Y)} }{\omega(\delta)}\leq \varepsilon.
	$$
	We have shown that $\|f-g\|_{\dot{C}^{0,\omega}(X,Y)} \leq \varepsilon.$

	Now, if $x\in X=c_0(\mathcal{A}),$ then (by definition) there exists a finite subset $S = S_x$ such that if $\alpha \in \mathcal{A} \setminus S,$ then $|x_\alpha| \leq r/2.$ Then, for $z \in B(x,r/2)$ and for each $\alpha \in \mathcal{A} \setminus S$ it follows that $|z_\alpha| \leq r,$ implying that $\varphi(z_\alpha)=0.$ It follows for $z\in B(x,r/2)$ that $\phi(z)=\phi(P_S(z))$ and hence that $g(z) = g(P_S(z)).$

	For the second part, suppose there exists $R>0$ so that $f(x)=0$ for all $\|x\| \geq R.$ If $\|z\| \geq R+r,$ then $|z_{\alpha_0} | \geq R + r$ for some $\alpha_0 \in \mathcal{A}.$ Consequently
	$$
	\|\phi(z)\| = \sup_{\alpha\in \mathcal{A}} |\varphi(z_\alpha)| \geq | \varphi(z_{\alpha_0}) | \geq |z_{\alpha_0}|-r \geq R,
	$$
	and thus $g(z)=f(\phi(z))=0.$
	Finally, if $f$ is Lipschitz, then clearly $g= f \circ \phi$ is Lipschitz.

\end{proof}

\begin{proof}[Proof of Theorem \ref{thm:main:c0YApprox}]
	
	In the proof of part (i), by Lemma \ref{lem:preliminaryforc0}, given a fixed $r>0,$ it is enough to approximate functions $g\in \VC^{0,\omega}_{\op{small}}(X,Y)$ with the property that for every $x\in X,$ there exists finite $ S_x \subset \mathcal{A}$ such that $g(z)=g(P_{S_x}(z)),$ for all $z\in B(x,2r).$

	Given $\varepsilon,$ let $\delta>0$ be so that 
	\begin{equation}\label{eq:mainc0theoremchoicedelta}
		\|g(u)-g(v)\|_Y \leq \varepsilon \omega(\|u-v\|) , \quad \text{whenever} \quad \|u-v\| \leq \delta.
	\end{equation}
	
	Moreover, for $0<\eta<\min(r, \varepsilon)$ small enough (more precisely, so that $\omega(\eta) \leq \varepsilon \omega(\delta)/ (1+\|g\|_{\dot{C}^{0,\omega}})),$ let $\theta: \R \to \R$ be an even $C^\infty$ smooth non-negative function such that $\int_{\R} \theta=1$ and $\theta = 0$ on $\R \setminus  [ -  \eta,  \eta].$ For every finite set $F\subset \mathcal{A},$ and $x\in X,$ we define the Bochner integral
	\begin{align}\label{eq:expressionc0mollification}
		h_F(x)
		:=\int_{[-\eta,\eta]^{|F|}} g \Big ( x- \sum_{\alpha \in F} t_\alpha e_\alpha \Big ) \prod_{\alpha \in F} \theta(t_\alpha)  \ud  \lambda_{|F|}(t),
	\end{align}
where $\lambda_{|F|}$ is the Lebesgue measure on $\R^{|F|}.$ 
That the integral $h_F(x)$ is well-defined for every $x\in X$ can be justified with the same arguments we used in the proof of Theorem \ref{thm:main:Euc}. Namely, if $Q:=[-\eta,\eta]^{|F|},$ the function
$$
Q \ni t \mapsto \sigma_x(t):=g \Big ( x- \sum_{\alpha \in F} t_\alpha e_\alpha \Big ) \prod_{\alpha \in F} \theta(t_\alpha) 
$$
is uniformly continuous $Q \to Y$, by the compactness of $Q$ and the continuity of the functions $g$ and $\theta.$ For each $y^* \in Y^*,$ the mapping $Q \to \R$ given by $Q \ni t \mapsto y^*( \sigma_x(t))$ is continuous, and thus $\lambda_{|F|}$-measurable over $Q.$ By continuity of $\sigma_x$, the image $\sigma_x(Q)$ is compact, and so $\overline{\mathrm{span}}(\sigma_x(Q))$ is separable. By \cite[Proposition 5.1]{BenyaminiLindenstrauss}, this shows the measurability of $\sigma_x$ in $Q.$ Also, by \cite[Proposition 5.2]{BenyaminiLindenstrauss} $\sigma_x$ is Bochner integrable if $\int_Q \| \sigma_x(t) \|_Y \ud  \lambda_{|F|}(t) < \infty,$ which is true because $\theta$ is a test function, and $g,$ being in $\dot{C}^{0,\omega}(X,Y),$ is bounded on bounded sets.

Because $g(z)=g(P_{S_x}(z))$ for all $z\in B(x,2r)$, the fact that $\eta < r$ implies that $h_{S_x}=h_{S_x} \circ P_{S_x}$ on the ball $B(x,r).$ Therefore
	$$
	h_{S_x}(z)  =\int_{[-\eta,\eta]^{|S_x|}} g \Big ( P_{S_x}(z)- \sum_{\alpha \in S_x} t_\alpha e_\alpha \Big ) \prod_{\alpha \in S_x} \theta(t_\alpha)  \ud \lambda_{|S_x|}(t), \quad z\in B(x,r).
	$$  
	Since $g\in  \dot{C}^{0,\omega}(X,Y)$ and $P_{S_x}$ is $1$-Lipschitz, the function $z\mapsto (g \circ P_{S_x})(z)$ belongs to $\dot{C}^{0,\omega}(X,Y)$, implying that $z\mapsto (g \circ P_{S_x})(z)$ is uniformly continuous on $X.$ Therefore, by the previous formula, $B(x,r) \ni z \mapsto h_{S_x}(z)$ is a finite dimensional smooth mollification (in the Bochner sense) with a uniformly continuous mapping.  As we explained in the proof of Theorem \ref{thm:main:Euc}, this guarantees that $h_{S_x} \in C^\infty(B(x,r),Y).$ Moreover, using Fubini's theorem, one can easily check that if $F\subset \mathcal{A}$ is finite and $F\supset S_x,$ then also $h_F=h_{S_x}$ on $B(x,r).$ This enables us to define the desired mapping $h:X \to Y$ in the following manner. Given $x\in X,$ let $S_x \subset \mathcal{A}$ be the subset described above, and define $h(x):=h_{S_x}(x).$ Moreover, ordering the collection of finite subsets $\mathcal{F}(\mathcal{A})$ of $\mathcal{A}$ by inclusion, the pointwise limit of the net $\lbrace h_S \rbrace_{S\in \mathcal{F}(\mathcal{A})}$ is exactly $h.$

Now, for every fixed ball $B(x,r)$ and $z\in B(x,r)$, let $F= S_x \cup S_z$. By the definition of $h$ and the above mentioned properties, we derive
$$
h(z)=h_{S_z}(z)=h_{F}(z)= h_{S_x}(z).
$$
Since $z\mapsto h_{S_x}(z)$ is of class $C^\infty(B(x,r),Y)$, the above yields $h\in C^\infty(X,Y).$ All these properties were stated and proved in \cite[Lemma 6]{HajJoh09}.

	We next show that $h\in \VC_{\op{small}}^{0,\omega}(X,Y).$ Given any two points $x,z\in X,$ let $S_x, S_z\subset\mathcal{A}$ be the finite subsets associated with $x$ and $z,$ and set $S=S_x \cup S_z.$ Since $h(x)=h_S(x)$ and $h(z)=h_S(z)$ we have
	\begin{equation}\label{eq:c0verificationVCsmall}
		\begin{split}
			\|h(x)-h(z)\|_Y & = \Big \| \int_{[-\eta,\eta]^{|S|}} \Big ( g \big ( x- \sum_{\alpha \in S}  t_\alpha e_\alpha  \big )-g \big ( z- \sum_{\alpha \in S}  t_\alpha e_\alpha  \big ) \Big)  \prod_{\alpha \in S} \theta(t_\alpha)  \ud \lambda_{|S|}(t) \Big \|_Y \\
			& \leq \int_{[-\eta,\eta]^{|S|}} \Big \|  g \big ( x- \sum_{\alpha\in S}  t_\alpha e_\alpha  \big )-g \big ( z- \sum_{\alpha\in S}  t_\alpha e_\alpha  \big ) \Big \|_Y   \prod_{\alpha \in S} \theta(t_\alpha)  \ud \lambda_{|S|}(t) \\
			&\leq \sup_{\substack{u,v \in X, \\ \|u-v\|=\|x-z\|}}\|g(u)-g(v)\|_Y, 
		\end{split}
	\end{equation}
	bearing in mind that $\int_{[-\eta,\eta]^{|S|}} \prod_{\alpha \in S} \theta(t_\alpha)  \ud \lambda_{|S|}(t) =1.$  It follows immediately from the bound \eqref{eq:c0verificationVCsmall} that $\|h\|_{\dot{C}^{0,\omega}(X,Y)} \leq \|g\|_{\dot{C}^{0,\omega}(X,Y)}$ and $h\in \VC_{\op{small}}^{0,\omega}(X,Y).$ Now if $\|x-z\| \leq \delta,$ then together \eqref{eq:c0verificationVCsmall} and \eqref{eq:mainc0theoremchoicedelta} give
	\begin{align*}
		&\|(g-h)(x)-(g-h)(z)\|_Y   \leq \|g(x)-g(z)\|_Y + \| h(x)-h(z)\|_Y   \leq 2 \varepsilon \omega(\|x-z\|).
	\end{align*}
	Hence to verify $\|g-h\|_{\dot{C}^{0,\omega}(X,Y)} \leq 2\varepsilon,$ it remains to check the case $\|x-z\| \geq \delta.$ To do so, let us see that $\sup_X \|g-h\| \leq  \varepsilon \omega(\delta).$ Indeed, if $x\in X$ and $S=S_x,$ then again $h(x) = h_{S_x}(x),$ and we obtain 
	\begin{align*}
		\|h(x)-g(x)\|_Y & \leq   \int_{[-\eta,\eta]^{|S|}} \Big \|g \Big ( x- \sum_{\alpha\in S} t_\alpha e_\alpha \Big )-g(x) \Big\|_Y \prod_{\alpha \in S} \theta(t_\alpha)  \ud \lambda_{|S|}(t) \\
		& \leq \|g\|_{\dot{C}^{0,\omega}(X,Y)}\int_{[-\eta,\eta]^{|S|}} \omega \Big ( \Big \| \sum_{\alpha \in S} t_\alpha e_\alpha \Big \| \Big )  \prod_{\alpha \in S} \theta(t_\alpha)  \ud \lambda_{|S|}(t) \\
		&\leq \|g\|_{\dot{C}^{0,\omega}(X,Y)} \omega(\eta) \leq \varepsilon \omega(\delta),
	\end{align*}
	where we chose $\eta$ small enough in the beginning of the proof for the last bound to hold.
	Now, if $x,z\in X$ are such that $\|x-z\| \geq \delta,$ then
	$$
	\|(g-h)(x)-(g-h)(z)\| \leq 2 \sup_X \|g-h\| \leq 2 \varepsilon \omega(\delta) \leq 2\varepsilon \omega(\|x-z\|).
	$$
	We conclude that $\|g-h\|_{\dot{C}^{0,\omega}(X,Y)} \leq 2\varepsilon$ and complete the proof of Theorem \ref{thm:main:c0YApprox}(i).

	For part (ii), given $f\in \VC^{0,\omega}(X,Y),$ by Theorem \ref{thm:approx:bs}, we approximate $f$, in the $\dot{C}^{0,\omega}(X,Y)$-seminorm by a $g\in \VC^{0,\omega}_{\op{small}}(X,Y)$ with bounded support. By Lemma \ref{lem:preliminaryforc0}, we can assume that for every $x\in X$ there exists a finite $S = S_x$ such that $g=g\circ P_S$ on $B(x,r),$ and for a uniform fixed $r>0.$ Let us next check that the approximation $h$ from part (i) has bounded support. Let $R>0$ be so that $g$ is zero on $X\setminus B(0,R).$ Assume that $\|x \| \geq R + \eta.$  There exists a finite set $S=S_x \subset \mathcal{A}$ so that $h(x)=h_S(x)$. Now, for $t=(t_\alpha)_{\alpha \in S} \in [-\eta,\eta]^{|S|},$ there holds that
	$$
	\Big \| x- \sum_{\alpha \in S} t_\alpha e_\alpha \Big \| \geq \|x\|- \Big \| \sum_{\alpha \in S} t_\alpha e_\alpha \Big \| \geq \|x\| -  \eta \geq R$$
	and thus
	$$
	h(x)=h_S(x)=   \int_{[-\eta,\eta]^{|S|}} g \Big ( x- \sum_{\alpha \in S} t_\alpha e_\alpha \Big ) \prod_{\alpha \in S} \theta(t_\alpha) \ud \lambda_{|S|}(t) =0.
	$$
	The proof of Theorem \ref{thm:main:c0YApprox}(ii) is now complete.

	\smallskip

	Finally, in the particular case $Y=\R,$ let us establish the desired Lipschitz approximations. Let $f\in \VC^{0,\omega}_{\op{small}}(X)$ be the function to be approximated. By Theorem \ref{thm:main:LipApprox} we can assume that $f$ is Lipschitz, and by Lemma \ref{lem:preliminaryforc0}, there is a Lipschitz approximation $g$ of $f$, with $g$ locally depending only on a finite number of coordinates around each ball of fixed radius $r>0.$ Repeating the construction of $h \in C^\infty(X,)$, we immediately see from the estimate \eqref{eq:c0verificationVCsmall}, that $h$ is Lipschitz, and in fact $\lip(h) \leq \lip(g).$ 
 Because any Lipschitz mapping in $\dot{C}^{0,\omega}(X)$ belongs to $\VC_{\op{small}}^{0,\omega}(X),$ by the condition \eqref{eq:mod:coer1}, the identity in $(i)$ now follows at once.

	\smallskip
	
	As for approximation of an $f\in \VC^{0,\omega}(X)$, we apply Theorem \ref{thm:main:LipApprox} to reduce matters to $f\in \lip_{\op{bs}}(X)$. Then the function $g$ of Lemma \ref{lem:preliminaryforc0} can be taken Lipschitz and with bounded support, and so can the function $h$ we constructed in (i) and (ii) of the present theorem. Thus $h\in  C^\infty_{\op{bs}} (X)\cap \lip(X)$ approximates $f$ in the $\dot{C}^{0,\omega}(X)$ seminorm. Because any Lipschitz mapping with bounded support belongs to $\VC^{0,\omega}(X),$ the identity in (ii) now follows at once. 
	
\end{proof}

For the particular case of $\mathcal{A}= \lbrace 1, \ldots, n \rbrace,$ $c_0(\mathcal{A})$ becomes $\R^n$ with the supremum norm, which is of course equivalent with the usual Euclidean setting. In $\R^n$ functions with bounded support have compact support and thus Theorem \ref{thm:main:c0YApprox} has the following consequence. 
\begin{corollary}
	Let $Y$ be an arbitrary Banach space. Then:
	\begin{itemize}
		\item[(i)] $\overline{C^{\infty}(\R^n,Y) \cap \dot{C}^{0,\omega}(\R^n,Y)}^{\dot{C}^{0,\omega}(\R^n,Y)}  =\VC^{0,\omega}_{\op{small}}(\R^n,Y),$
		\item[(ii)] $\overline{C^{\infty}_{c}(\R^n,Y)}^{\dot{C}^{0,\omega}(\R^n,Y)} = \VC^{0,\omega}(\R^n,Y).$
	\end{itemize}
	
\end{corollary}

\appendix

\section{Comparison of $\dot C^{0,\omega}$ and uniform convergence}\label{appendix:comparison}

We begin with a remark concerning the two types of convergences we are dealing with. We will refer to the norm $||| \cdot |||_{C^{0,\omega}(X,Y)}$ defined in \eqref{eq:BanachNorm}.

\begin{remark}\label{remark:A1}
{\em Let $X$ and $Y$ be normed spaces and $\omega$ any modulus of continuity. Suppose that a sequence of functions $(f_n)_n \subset   \dot C^{0,\omega}(X,Y) $ converges to $f,$ with respect to the norm $|||\cdot|||_{ C^{0,\omega}(X,Y)}.$ Then also $f_n \to f$ uniformly on bounded subsets of $X.$

Indeed, let $B$ a bounded subset of $X$ with $0 \in B.$ Given $\varepsilon>0,$ let $N$ be so that $|||f_n-f|||_{  C^{0,\omega}(X,Y)} \leq \varepsilon/(1+ \omega(\diam
(B)))$ for all $n \geq N.$ Then, for those $n\geq N,$ and $x \in B,$ we can write 
\begin{align*}
|f(x)-f_n(x)| &  \leq |f(0)-f_n(0)|  + |(f-f_n)(x)-(f-f_n)(0)| \\
& \leq \left(1+\omega(\|x\|) \right) |||f_n-f|||_{  C^{0,\omega}(X,Y)} \leq \left( 1+ \omega(\diam(B)) \right) |||f_n-f|||_{  C^{0,\omega}(X,Y)} \leq \varepsilon.
\end{align*}
}
\end{remark}

However, the convergence in $\dot C^{0,\omega}$ does not imply \textit{global} uniform convergence, even if we assume that the pertinent sequence of functions is globally uniformly bounded, as the next example shows. 

\begin{example}\label{example:2}
{\em Let $\omega$ be of the form $\omega(t)=t^\alpha,$ with $0 < \alpha \leq 1.$ Define $f_n : \R \to \R$ as the \textit{continuous} function given by $f_n=0$ on $(-\infty,0] \cup [2n, +\infty),$ $f_n(n)=1$, and $f_n$ affine on both $[0,n]$ and $[n,2n].$ It is immediate that $\|f_n\|_\infty \leq 1$ for all $n,$ and that $(f_n)_n$ does not converge to $0$ uniformly on $\R.$ However, one has $ f_n  \to 0$ in the $\dot C^{0,\omega}(\R)$ sense. This easily follows after observing that
$$
 \sup_{n \leq s<t \leq 2n } \frac{|f_n(t)-f_n(s)|}{\omega(|t-s|)} =\sup_{0 \leq s<t \leq n } \frac{|f_n(t)-f_n(s)|}{\omega(|t-s|)} = \sup_{0 \leq s<t \leq n } \frac{1}{n}\frac{t-s}{(t-s)^{\alpha}} = \frac{1}{n^\alpha}.
$$

}
\end{example}

Finally, we next show that the converse to Remark \ref{remark:A1} is false, by providing a sequence $(f_n)_n \subset \dot C^{0,\omega}$ and $f \in  \dot C^{0,\omega}$ so that $f_n$ converges uniformly on $\R,$ and $f_n$ does not converge to $f$ in the $\| \cdot \|_{\dot C^{0,\omega}}$ norm on bounded subsets.  
\begin{example}\label{example:3}
{\em  Define $f=0$ on $\R$ and the \textit{continuous} functions $f_n: \R \to [0,1]$ given by
 \begin{align*} 
		f_n(t)= \begin{cases}
			1/\sqrt{n},\qquad & t=0, \\
			\text{affine},\qquad &  t\in [0,1/n],\\
			0,\qquad & t \geq 1/n, \\
			f_n(-t), \qquad & t \leq 0.  
		\end{cases}
	\end{align*}
Then $\sup_{t\in \R} |f_n(t)| = 1/\sqrt{n}, $ yielding $\|f_n-f\|_\infty \to 0.$ Also, each $f_n$ is $\sqrt{n}$-Lipschitz:
$$
\frac{|f_n(0)-f_n(1/n)|}{|0-1/n|} = \frac{1/\sqrt{n}}{1/n}= \sqrt{n}. 
$$
Because $f_n $ is bounded, then $f_n \in \dot C^{0,\omega}(\R)$ for all moduli $\omega$ that satisfy $\lim_{t \to 0} t/\omega(t) < \infty.$

However, if $\omega(t)=t^\alpha $ with $\alpha \in (1/2,1],$ then $\lim_n \|f_n-f\|_{\dot C^{0,\omega}([0,1])} = \infty,$ since
$$
\lim_n \|f_n-f\|_{\dot C^{0,\omega}([0,1])}  \geq \lim_n \frac{|f_n(0)-f_n(1/n)|}{\omega(|0-1/n|)} =\lim_n n^{\alpha-\frac{1}{2}} = \infty. 
$$}
\end{example}

\bibliographystyle{abbrv}
\bibliography{references_AiM240611_MudOik}

\end{document}